\documentclass[5p,number,letter]{elsarticle}
\pdfoutput=1

\usepackage{bbm}
\usepackage{euscript}
\usepackage{amsmath}
\usepackage{color}

\def\cal{\EuScript}
\def\mathscr{\mathcal}

\newtheorem{proposition}{Proposition}
\newtheorem{theorem}{Theorem}
\newtheorem{corollary}{Corollary}
\newdefinition{example}{Example}
\newproof{proof}{Proof}

\def\eps{\varepsilon}
\def\wh#1{\widehat{#1}}
\def\qed{\rule{1.5mm}{1.5mm}}

\mathcode`\@="8000
{\catcode`\@=\active \gdef@{\mkern1mu}} 

\def\C{\mathbbm{C}}

\def\R{\mathbbm{R}}

\def\dop{{\rm d}}
\def\eop{{\rm e}}

\def\Rnn{\R^{n\times n}}

\def\Cn{\C^n}

\def\Ckk{\C^{k\times k}}

\def\Cnk{\C^{n\times k}}

\def\Cnn{\C^{n\times n}}

\def\CNN{\C^{N\times N}\kern-2pt}

  \def\BA{{\bf A}} 
\def\Bb{{\bf b}}  \def\BB{{\bf B}} 
\def\Bc{{\bf c}}  \def\BC{{\bf C}} 
  \def\BD{{\bf D}} 
\def\Be{{\bf e}}  \def\BE{{\bf E}}

  \def\BH{{\bf H}} 
  \def\BI{{\bf I}} 
   
   \def\CK{{\cal K}}

  \def\BT{{\bf T}} 
\def\Bu{{\bf u}}   
\def\Bv{{\bf v}}  \def\BV{{\bf V}} \def\CV{{\cal V}}
\def\Bw{{\bf w}}  \def\BW{{\bf W}} \def\CW{{\cal W}}
\def\Bx{{\bf x}}   
\def\By{{\bf y}}

\def\Bzero{{\bf 0}}
\def\BAs{\BA{\kern-1pt}}      
\def\BVs{\BV{\kern-1pt}}      
\def\BWs{\BW{\kern-1pt}}      

\def\xhat{\widehat{\Bx}}
\def\yhat{\widehat{y}}

\def\VAV{\BV^*\kern-1.5pt\BA\BV}
\def\WAV{\BW^*\kern-1.5pt\BA\BV}
\def\VkAVk{\BV^*_{\kern-2pt k}\kern-1pt\BA\BV_{\kern-2pt k}^{}}
\def\WkAVk{\BW^{@*}_{\kern-2pt k}\kern-.5pt\BA\BV_{\kern-2pt k}^{}}

\title{Unstable modes in projection-based reduced-order models: \\
How many can there be, and what do they tell you?}

\author{Mark Embree\fnref{fn1}}
\ead{embree@vt.edu}
\address{Department of Mathematics and
Computational Modeling and Data Analytics Division,
Academy of Integrated Science\\
Virginia Tech, Blacksburg, VA 24061, USA}

\fntext[fn1]{Supported through U.S. National Science Foundation
grant DMS-1720257.}

\begin{document}

\begin{abstract}
Projection methods provide an appealing way to construct reduced-order
models of large-scale linear dynamical systems: 
they are intuitively motivated and fairly easy to compute.
Unfortunately, the resulting reduced models need not inherit the stability
of the original system.
How many unstable modes can these reduced models have?
This note investigates this question, using theory originally motivated by
iterative methods for linear algebraic systems and eigenvalue problems,
and illustrating the theory with a number of small examples.
From these results follow rigorous upper bounds on the number of unstable
modes in reduced models generated via orthogonal projection, 
for both continuous- and discrete-time systems.
Can anything be learned from the unstable modes in reduced-order models?
Several examples illustrate how such instability can helpfully signal transient 
growth in the original system.
\end{abstract}

\maketitle

\section{Introduction}
Reduced-order models are an enabling technology for simulation and design.
One seeks simple low-order models that mimic the dynamics of a system with a 
high-dimensional state space.
Asymptotic stability is the most fundamental property the reduced system should capture,
but several popular algorithms can construct unstable re\-duc\-ed-order models for stable systems.

In this note, 
we investigate the potential instability of
reduced-order models (ROMs) derived from projection methods.
For simplicity of presentation, consider
the standard con\-tin\-uous-time, 
single-input, single-output (SISO) linear system
\begin{eqnarray}
\dot{\Bx}(t) &\!\!\!=\!\!\!& \BA\Bx(t) + \Bb @u(t) \label{eq:xdot} \\
y(t)       &\!\!\!=\!\!\!& \Bc^*\Bx(t) + d @u(t); \label{eq:y}
\end{eqnarray}
here $\BA\in\Cnn$, $\Bb,\Bc\in\C^{n\times 1}$, and $d\in\C$  
($\cdot^*$ denotes the conjugate-transpose).
For details about projection-based reduced-order modeling,
see, e.g.,~\cite{Ant05b}.  

Orthogonal projection algorithms restrict the state to evolve 
in the $k$-dimensional subspace $\CV \subset \Cn$,
then close the system by imposing a \emph{Galerkin condition}:
the reduced system's misfit should be orthogonal to the subspace $\CV$. 
While the choice of $\CV$ is crucial to the quality and properties
of the resulting ROM, many of the results 
we discuss apply to any choice of $\CV$
(including $\CV$ derived from moment-matching
reduction of multi-input, multi-output (MIMO) systems).

Let the columns of $\BV\in\Cnk$ form an orthonormal basis for $\CV$, so
$\BV^*\BV = \BI$.
To reduce the dimension of the system~(\ref{eq:xdot})--(\ref{eq:y}),
approximate $\Bx(t) \approx \BV \xhat(t) \in \CV$.\ \ 
One cannot simply replace $\Bx(t)$ by $\BV@\xhat(t)$ in~(\ref{eq:xdot}),
since in general $\BA\BV@\xhat(t) + \Bb @u(t) \not \in \CV$.\ \ 
To obtain a well-determined equation, impose the Galerkin condition
\begin{equation} \label{eq:Galerkin}
 \BV^*\Big(\BV@\dot{\xhat}(t) - \big(\BA\BV\xhat(t) + \Bb @u(t)\big)\Big) = \Bzero,
\end{equation}
which yields the reduced system
\begin{eqnarray}
\dot{\xhat}(t) &=& (\VAV)@\xhat(t) + (\BV^*\Bb) @u(t) \label{eq:xhatdot} \\
\yhat(t)       &=& (\Bc^*\BV)@\xhat(t) + d@ u(t). \label{eq:yhat}
\end{eqnarray}
(Oblique projection methods impose the orthogonality in (\ref{eq:Galerkin})
against a different subspace; see Section~\ref{sec:oblique} for details.)

For an effective ROM, one seeks a subspace $\CV$
of smallest possible dimension $k$ for which the reduced
output $\yhat(t)$ mimics the true output $y(t)$,
i.e., to make $\|y-\yhat\|$ small in an appropriate norm.
Taking $\CV$ to be a Krylov subspace gives particularly appealing properties,
but the framework we describe also applies to the Galerkin proper orthogonal
decomposition (POD) method (see, e.g., \cite{KV02,RP03}) 
applied to a linear system; one might extrapolate some insight 
about the behavior of POD for nonlinear systems.

\subsection{Moment-matching projection}
The degree-$k$ Krylov subspace generated by the matrix $\BA\in\Cnn$ and
vector $\Bb\in\Cn$ is
\begin{equation} \label{eq:kry}
   \CK_k(\BA,\Bb) = {\rm span}\{\Bb, \BA\Bb, \ldots, \BAs^{k-1}\Bb\}.
\end{equation}
Under mild conditions, ${\rm dim}(\CK_k(\BA,\Bb)) = k$.
If we take $\CV = \CK_k(\BA, \Bb)$ and let $\BV$ have 
orthonormal columns,  then \emph{the resulting ROM 
matches the first $k$ moments of the original model}:
\[ \Bc^*\BAs^s\kern1pt \Bb = (\Bc^* \BV)(\VAV)^s(\BV^*\Bb), 
  \quad s=0, \ldots, k-1.\]
That is, the first $k$ terms of the Taylor expansion of the
transfer function
\[ H(z) := \Bc^*(z\BI-\BA)^{-1}\Bb\]
match those of the reduced transfer function
\[ \wh{H}(z) := (\BV\Bc)^*(z\BI-\VAV)^{-1}(\BV^*\Bb)\]
when expanded about $z=\infty$~\cite[Section~11.2]{Ant05b}.
To better capture the frequency response about some finite point $\mu\in\C$,
one can instead select
\[ \CV = \CK_k((\mu@\BI-\BA)^{-1},\Bb)\]
to match moments at $\mu$.
(The oblique projection method based on the bi-Lanczos method addressed 
in Section~\ref{sec:oblique} gives ROMs that match twice as many moments 
as these orthogonal projection Krylov methods, but this extra measure of 
accuracy can come at the cost of numerical challenges and additional unstable 
modes.)

\smallskip
This elegant moment matching gives a compelling motivation for Krylov 
projection methods, but these techniques have a crucial weakness:
even when the matrix $\BA$ is stable (all eigenvalues in the left half-plane),
the reduced model $\VAV$ can have eigenvalues in the right half-plane:
the original system is asymptotically stable (all solutions to
$\dot\Bx(t) = \BA\Bx(t)$ converge to zero), but the ROM supports solutions 
that diverge as $t\to\infty$.

In this note we investigate this phenomenon, drawing on results that have been
developed to explain the behavior of iterative methods for the solution
of linear algebraic systems and eigenvalue problems.
After reviewing spectral properties associated with transient dynamics 
in Section~\ref{sec:spectral},
in Sections~\ref{sec:orth_cont} and~\ref{sec:orth_disc} we give upper bounds
on the number of unstable eigenvalues the reduced matrix $\VAV$ can have
for continuous- and discrete-time systems.
Sections~\ref{sec:orth_bad} and~\ref{sec:oblique} describe adversarial constructions
that can produce many unstable modes for orthogonal and oblique projection methods.
Throughout these sections, we illustrate theory with toy examples that are easy to analyze.

Are unstable modes merely a scourge?
In Section~\ref{sec:unstable} we argue that unstable modes
can actually give valuable insight about the transient behavior of 
the original system, and efforts to tame these unstable modes can
result in stable ROMs that fail to accurately model the short-term
behavior of the original system.

Throughout, we use $\Be_j$ to denote the $j$th column of the identity matrix
(whose dimension should be clear from the context), and, unless noted otherwise,
 $\|\cdot\|$ to denote the vector 2-norm and the associated matrix norm.

\section{Spectral preliminaries} \label{sec:spectral} 

Since we seek to understand the  asymptotic and transient
behavior of dynamical systems (both full- and reduced-order models),
we review some helpful quantities associated with the spectrum.  
For more detailed descriptions and illustrative examples, see~\cite{TE05plain}.

Denote the spectrum (set of eigenvalues) of $\BA$ by
\[ \sigma(\BA) := \{\lambda_1, \ldots, \lambda_n\}.\]
Two scalar quantities dictate 
the asymptotic stability of continuous- and discrete-time systems,
the \emph{spectral abscissa} $\alpha(\BA)$ and the 
\emph{spectral radius} $\rho(\BA)$:
\[ \alpha(\BA) := \max_{\lambda \in \sigma(\BA)} {\rm Re}\, \lambda, \qquad
   \rho(\BA)   := \max_{\lambda \in \sigma(\BA)} |\lambda|.
\]
The \emph{numerical range} (or \emph{field of values}) 
\begin{equation} \label{eq:nr}
 W(\BA) := \{ \Bv^*\BA\Bv : \Bv\in\Cn, \|\Bv\| = 1\} 
\end{equation}
is a closed, convex subset of $\C$ that contains $\sigma(\BA)$; 
for details, see~\cite[Chapter~1]{HJ91}. 
We denote its maximal real extent and magnitude as the
\emph{numerical abscissa} $\omega(\BA)$ and the 
\emph{numerical radius} $\nu(\BA)$:
\begin{equation} \label{eq:nabsrad}
 \omega(\BA) := \max_{z\in W(\BA)} {\rm Re}\ z, \qquad
      \nu(\BA) := \max_{z\in W(\BA)} |z|.
\end{equation}
For any $\eps>0$, the \emph{$\eps$-pseudospectrum} of $\BA$,
\begin{eqnarray*}
\sigma_\eps(\BA) 
     &:=& \{ z\in \C:  \|(z\BI-\BA)^{-1}\|>1/\eps\} \\[.5em]
     &\kern3pt=& \{ z \in \sigma(\BA+\BE) \mbox{ for some $\BE\in\Cnn$} \\
     && \hspace*{10em} \mbox{with $\|\BE\|<\eps$}\},
\end{eqnarray*}
contains $\sigma(\BA)$, but also potentially points 
that are a distance much greater than $\eps$ from the spectrum. 
To analyze transient behavior of solutions to $\dot\Bx(t)=\BA\Bx(t)$, 
we can use the
\emph{$\eps$-pseudospectral abscissa} $\alpha_\eps(\BA)$ and the 
\emph{$\eps$-pseudospectral radius} $\rho_\eps(\BA)$:
\begin{equation} \label{eq:psabsrad}
 \alpha_\eps(\BA) := \max_{z \in \sigma_\eps(\BA)} {\rm Re}\, z, \qquad
   \rho_\eps(\BA)   := \max_{z \in \sigma_\eps(\BA)} |z|.
\end{equation}
A theorem of Stone (see~\cite[eq.~(17.9)]{TE05plain}) shows that the
$\eps$-pseudospectrum cannot be more than $\eps$ larger than the
numerical range:
\[ \sigma_\eps(\BA) \subseteq W(\BA) + \Delta_\eps,\]
where $\Delta_\eps = \{z\in\C:|z|<\eps\}$ is the open ball of
radius $\eps$.

(The definition of $\sigma_\eps(\BA)$ permits general perturbations $\BE\in\Cnn$.
If $\BA$ is real valued, $\BA\in\Rnn$, might one gain insight by restricting perturbations
to $\BE\in\Rnn$?  This question motivates the study of \emph{structured pseudospectra},
or \emph{spectral value sets}~\cite{HK93,Kar03,Rum06}.  Considering only real
perturbations can significantly reduce the set $\sigma_\eps(\BA)$, 
but cannot improve the condition number of any eigenvalue by more than a factor of $1/\sqrt{2}$~\cite{BK04}.
For analyzing transient behavior of a linear system,
Example~(49.9) in~\cite{TE05plain} shows that complex perturbations are necessary
to reveal the potential for transient growth of real-valued linear systems.)

We seek to use the sets $\sigma(\BA)$, $W(\BA)$, and $\sigma_\eps(\BA)$ 
to gain insight into projection-based ROMs.
A class of matrices is worth singling out for their clean properties:
a matrix is \emph{normal} provided $\BAs^*\BA = \BA\BAs^*$.
Equivalently, a normal matrix has a unitary basis of eigenvectors.
This latter property makes it easy to show that
\[ \mbox{$\BA$ normal} \Longrightarrow 
      \left\{ \begin{array}{ll}
               W(\BA) = \mbox{ convex hull of $\sigma(\BA)$}; \\[.5em]
               \sigma_\eps(\BA) = \sigma(\BA) + \Delta_\eps. 
      \end{array}\right. \]
(Hermitian ($\BA=\BAs^*$), skew-Hermitian ($\BA = -\BAs^*$), 
and unitary ($\BAs^*\BA = \BI$) matrices are all normal.)
We will refer to the ``departure from normality'' as a 
gauge of how far a matrix is from the set of normal matrices.

\subsection{Potential for unstable modes}

We shall say that a continuous-time system is \emph{stable}
(i.e., \emph{asymptotically stable}) provided
\[ \alpha(\BA) < 0,\]
i.e., $\sigma(\BA)$ is contained in the open left half of the complex plane.
This condition implies that all solutions $\Bx(t) = \eop^{t\BA}\Bx(0)$ to $\dot\Bx(t) = \BA\Bx(t)$ converge to zero as $t\to\infty$.
Similarly, a discrete-time system is \emph{stable} provided
\[ \rho(\BA) < 1,\]
i.e., $\sigma(\BA)$ is contained in the open unit disk, so
all solutions $\Bx_k = \BAs^k \Bx_0$ to $\Bx_{k+1} = \BA\Bx_k$ 
converge to zero as $k\to\infty$.

Where can eigenvalues of $\VAV$ fall, relative to these spectral quantities 
associated with $\BA$?
We begin with a fundamental property of non-Hermitian 
eigenvalue approximation; cf.~\cite[Prop.~1.2.13]{HJ91}, \cite[Thm.~3.1]{MS96}.

\begin{proposition} \label{prop:fov}
Suppose the columns of $\BV\in\Cnk$ are orthonormal. Then 
\begin{equation} \label{eq:Wcontain}
 \sigma(\VAV) \subseteq W(\BA),
\end{equation}
and so any $\theta \in \sigma(\VAV)$ must satisfy
\[ {\rm Re}\, \theta \le \omega(\BA), \qquad |\theta| \le \nu(\BA).\]
\end{proposition}
The proof of~(\ref{eq:Wcontain}) is simple: 
if $\theta\in\sigma(\VAV)$, there exists a 
unit vector $\By$ such that $(\VAV)\By = \theta@\By$. 
Then $\|\BV\By\|^2 = \|\By\|^2 = \By^*\By=1$ 
since $\BV$ has orthonormal columns, and
\[ (\BV\By)^*\BA(\BV\By) = \By^*\VAV\By = \theta@@\By^*\By = \theta;\]
use the definition~(\ref{eq:nr}) to conclude that $\theta\in W(\BA)$.

\smallskip
It follows that if $W(\BA)$ is contained in the left half-plane 
(i.e., $\omega(\BA)<0$), 
then $\VAV$ is guaranteed to be stable.  
All stable normal matrices satisfy this property:
if $\sigma(\BA)$ is contained in the left half-plane, so too is 
its convex hull, which equals $W(\BA)$ for normal $\BA$.

\begin{proposition}
If $\BA$ is stable and normal, then $\VAV$ is stable for any 
choice of the subspace $\CV$.
\end{proposition}

Unfortunately, for many interesting stable models we find that
$\omega(\BA)>0$.
These are the matrices in which we are primarily concerned here.

\subsection{Transient behavior}

The numerical abscissa $\omega(\BA)$ does not simply bound the rightmost extent of 
$\theta\in\sigma(\VAV)$, as in Proposition~\ref{prop:fov};
it also signals whether solutions $\eop^{t\BA}\Bx(0)$ to 
$\dot{\Bx}(t) = \BA\Bx(t)$ can initially exhibit transient growth~\cite[Chapter~17]{TE05plain}:
\begin{equation} \label{eq:t0}
\max_{\|\Bx(0)\|=1} {\dop \over \dop t} \|\Bx(t)\| \bigg|_{t=0} 
= {\dop \over \dop t} \|\eop^{t\BA}\| \bigg|_{t=0} = \omega(\BA).
\end{equation}
The possibility for stable systems to grow on transient time scales has
important physical implications, especially for systems that arise
as linearizations of nonlinear systems; see~\cite[Part~V]{TE05plain} for examples
from fluid dynamics, and Section~\ref{sec:unstable} for an example involving
a nonlinear heat equation.

The formula~(\ref{eq:t0}) based on the numerical range 
describes the system's performance near $t=0$;
pseudospectra give insight into the maximum transient growth.
The simplest result gives a lower bound~\cite[Theorem~15.4]{TE05plain}:
\[ \sup_{t\ge 0} \|\eop^{t\BA}\| \ge {\alpha_\eps(\BA,\BE) \over \eps}\]
for all $\eps>0$.
The sets $W(\BA)$ and $\sigma_\eps(\BA)$ generalize to 
matrix pencils, informing the transient dynamics of 
differential--algebraic equations and descriptor systems~\cite{EK17}.

\section{An upper bound on unstable modes for orthogonal projection ROMs (continuous time case)} \label{sec:orth_cont}
Let $\BA$ be a stable matrix with eigenvalues $\lambda_1, \ldots, \lambda_n$
all satisfying ${\rm Re}\,\lambda_j < 0$, and let 
\[  \VAV \in \C^{k\times k} \]
denote an order-$k$ ROM constructed via orthogonal projection.
The columns of $\BV\in\Cnk$ are orthonormal, but 
we make no assumptions about the projection subspace 
${\rm range}(\BV)$; it could derive from a Krylov method,
POD, or any other algorithm.  
For example, while we introduced moment matching model reduction
for SISO systems, the results in this section also apply 
to moment matching for MIMO systems of the form
\begin{eqnarray*}
\dot{\Bx}(t) &\!\!\!=\!\!\!& \BA\Bx(t) + \BB @\Bu(t)  \\
\By(t)       &\!\!\!=\!\!\!& \BC\Bx(t) + \BD @\Bu(t),
\end{eqnarray*}
where ${\rm range}(\BV) = {\rm range}([\BB,\ \BA\BB,\ \ldots, \BAs^{\ell-1}\BB])$
is a block Krylov subspace of dimension $k$.

\smallskip
We begin with a simple example that shows how $W(\BA)$ can extend into
the right half-plane, even when $\BA$ is stable.
The Hermitian part of $\BA$,
\[ \BH := {\textstyle{1\over2}} (\BA+\BAs^*),\]
plays a critical role in stability theory.
Even when a non-Hermitian $\BA$ is stable, $\BH$ need not be.  
For example, for
\[ \BA = \left[\begin{array}{rr} -1 & 4 \\ 0 & -1\end{array}\right], \qquad
   \BH = \left[\begin{array}{rr} -1 & 2 \\ 2 & -1\end{array}\right], 
\] 
giving
\[ \sigma(\BA) = \{ -1, -1\}, \qquad \sigma(\BH) = \{-3, +1\};\]
and hence $\BH$ is not stable despite the stability of $\BA$.

\smallskip
Since $\BH$ is Hermitian, its eigenvalues are real.  
Label them in decreasing order as 
\[ \mu_1 \ge \mu_2 \ge \cdots \ge \mu_n.\]
Notice that $\mu_1 = \omega(\BA)$, the numerical abscissa~(\ref{eq:nabsrad})
that describes the rightmost extent of the numerical range $W(\BA)$.
To see this, take any $z\in W(\BA)$, for which there must exist some 
unit vector $\Bv\in\Cn$ such that $z = \Bv^*\BA\Bv$.  Then
\[ {\rm Re}\, z = {z+\overline{z} \over 2} 
                = \Bv^* \bigg({\BA+\BAs^*\over 2}\bigg) \Bv
                = \Bv^* \BH \Bv.\]
Thus the real part of any $z$ in the numerical range of $\BA$ is a
Rayleigh quotient for $\BH$.  
By the variational characterization of eigenvalues of Hermitian matrices,
\[ \mu_n \le {\rm Re}\,z \le \mu_1,\]
with equality attained when $\Bv$ is an eigenvector of
$\BH$ associated with $\mu_n$ or $\mu_1$;
see, e.g., \cite[Theorem 4.2.6]{HJ13}.

\smallskip
Via~(\ref{eq:t0}), the rightmost eigenvalue of $\BH$ gives insight 
into the initial behavior of solutions to $\dot\Bx(t) = \BA\Bx(t)$.
Only recently has it been appreciated that the \emph{interior
eigenvalues of $\BH$} help bound the eigenvalues of $\VAV$.
We first state a result from~\cite[Theorem~2.1]{CE12},
which was developed to support convergence analysis for the 
restarted Arnoldi method for computing 
eigenvalues of large matrices.
Theorem~\ref{thm:carden} establishes vertical strips in the complex plane 
where eigenvalues of $\VAV$ must fall. 
We follow the statement with a small example to illustrate its application.

\begin{theorem} \label{thm:carden}
Denote the eigenvalues of $\VAV \in \Ckk$ by
$\theta_1, \ldots, \theta_k$, labeled by decreasing real part:
\[ {\rm Re}\, \theta_1 \ge 
   {\rm Re}\, \theta_2 \ge  \cdots \ge
   {\rm Re}\, \theta_k.
\]
Let $M_{\pm j}$ denote the arithmetic mean 
of the $j$ largest and smallest eigenvalues of $\BH$,
\begin{eqnarray*}
   M_j &\!\!\!:=\!\!\!& {\mu_1 + \cdots + \mu_j \over j}, \rlap{\kern32.5pt$1\le j\le n$,} \\
   M_{-j} &\!\!\!:=\!\!\!& {\mu_{n-j+1} + \cdots + \mu_n \over j}, \rlap{\quad $1\le j\le n$,} 
\end{eqnarray*}
so $M_1 \ge M_2 \ge \cdots \ge M_n$ and 
$M_{-1} \le M_{-2} \le \cdots \le M_{-n}$.
Then for $1\le j\le k$,
\begin{equation} \label{eq:carden} 
    M_{-k+j-1} \ \le\  {\rm Re}\, \theta_j\  \le\  M_j.
\end{equation}
\end{theorem}

Thus, the $j$th rightmost eigenvalue $\theta_j$ of $\VAV$ must fall in the
intersection of $W(\BA)$ with the vertical strip 
\[ M_{-k+j-1} \le {\rm Re}\, z \le M_j.\]

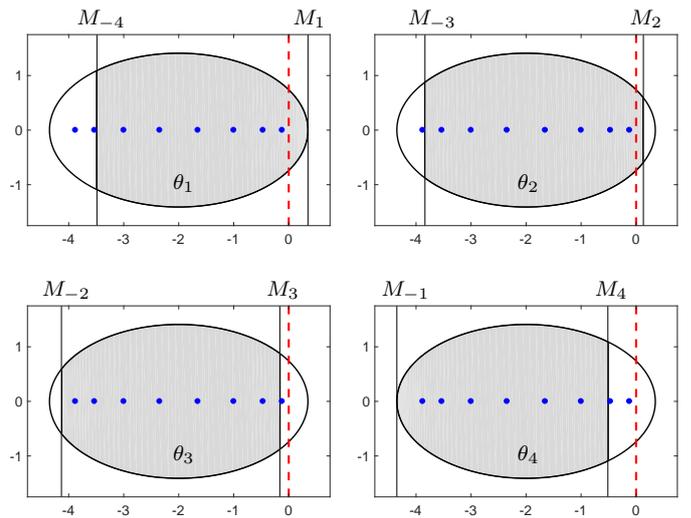
\begin{figure}[b!]
\vspace*{1.5em}
\includegraphics[scale=0.26]{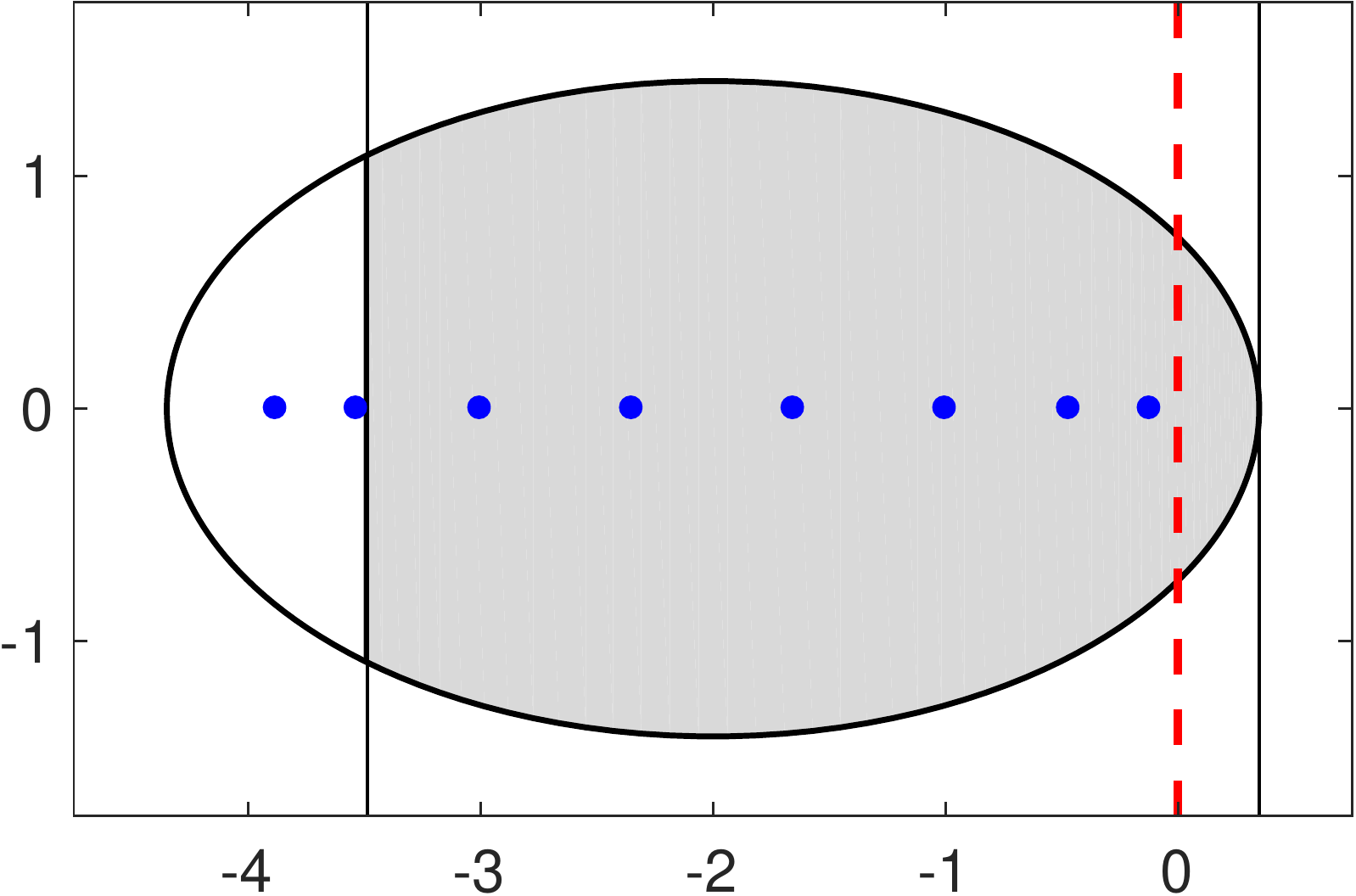}\quad
\includegraphics[scale=0.26]{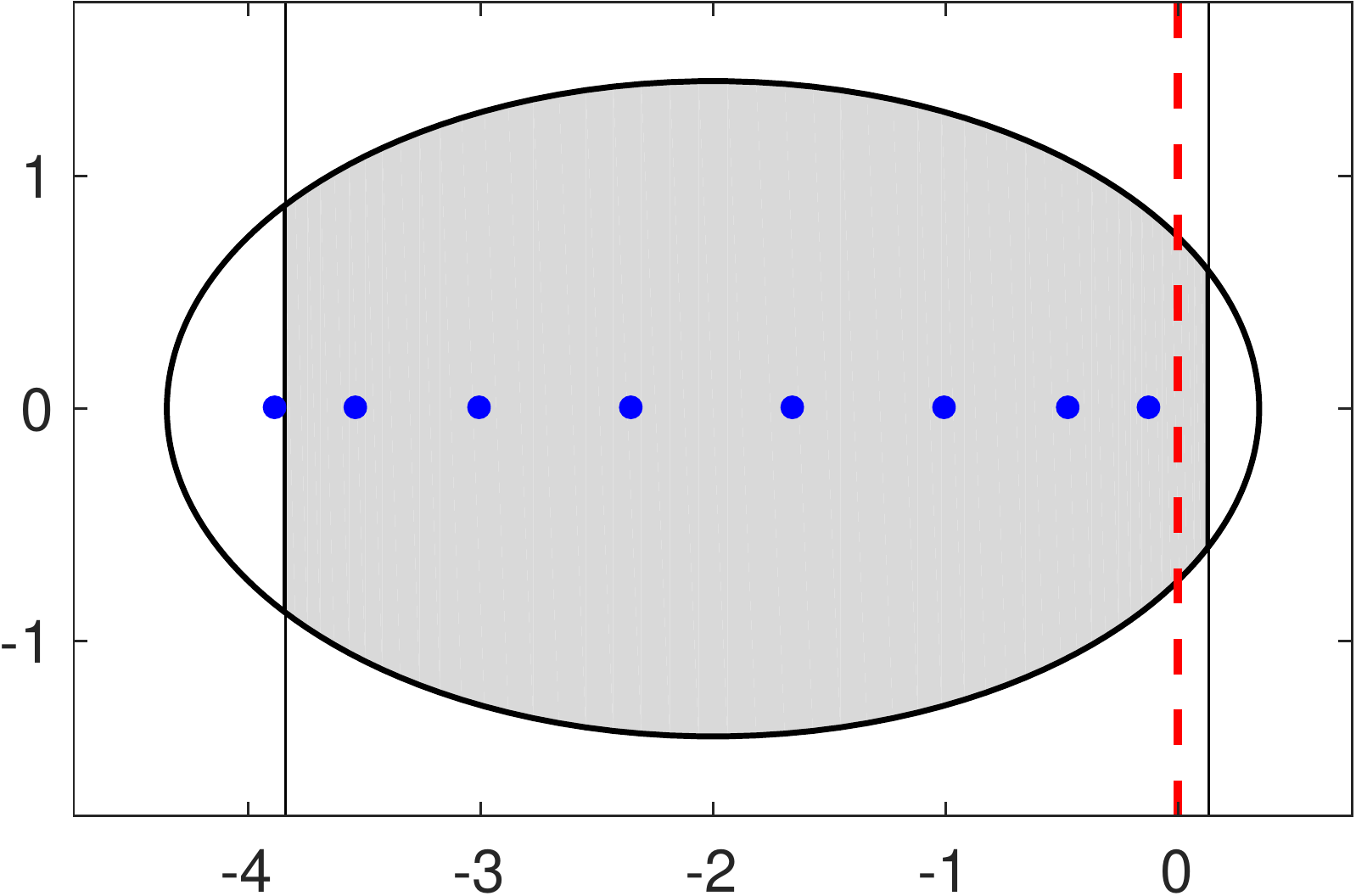}

\begin{picture}(0,0)
\put(25,95){\footnotesize $M_{-4}$}
\put(106,95){\footnotesize $M_1$}
\put(61,33){\footnotesize $\theta_1$}
\put(149,95){\footnotesize $M_{-3}$}
\put(232,95){\footnotesize $M_2$}
\put(190,33){\footnotesize $\theta_2$}
\end{picture}

\vspace*{1em}
\includegraphics[scale=0.26]{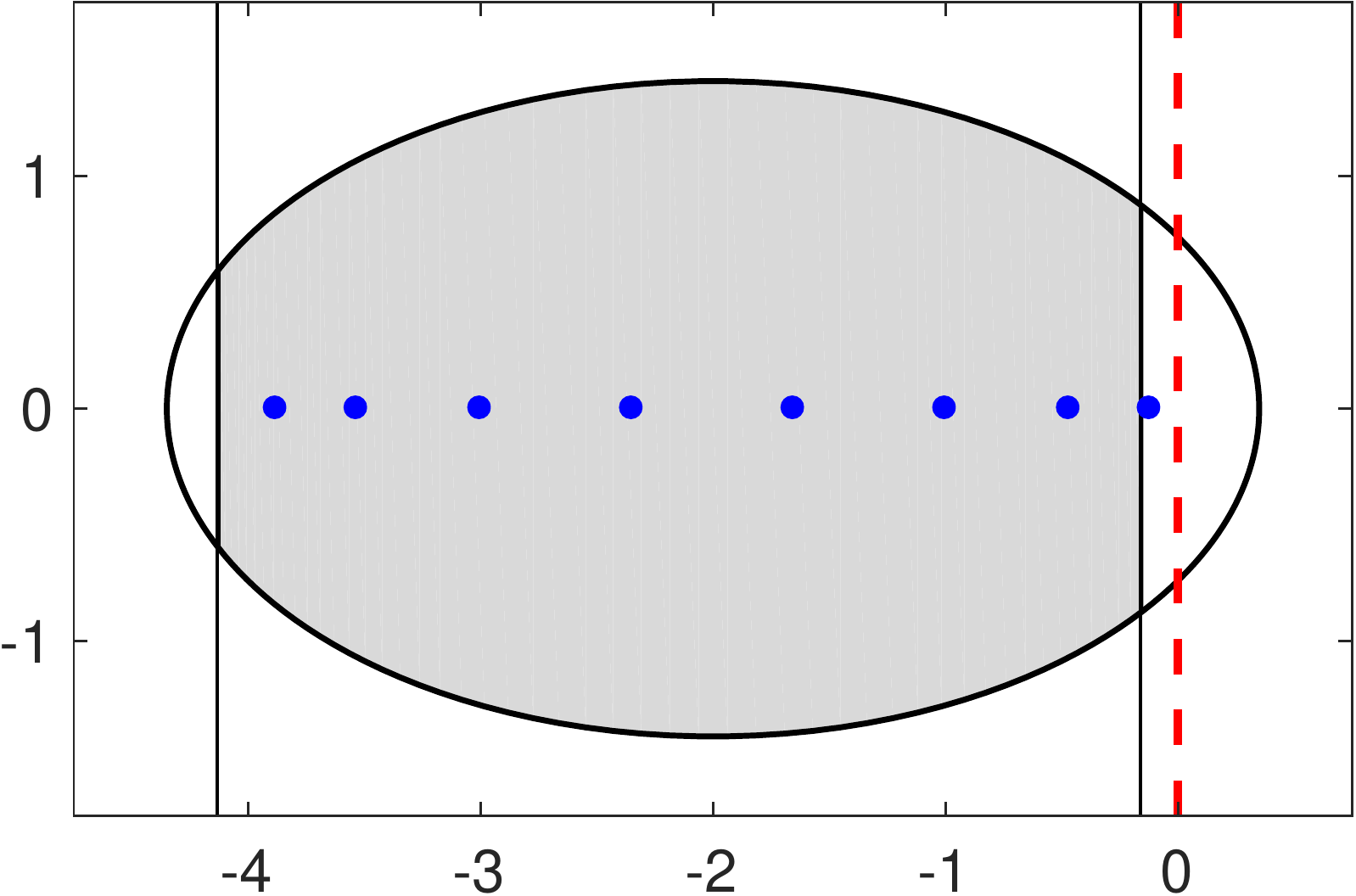}\quad
\includegraphics[scale=0.26]{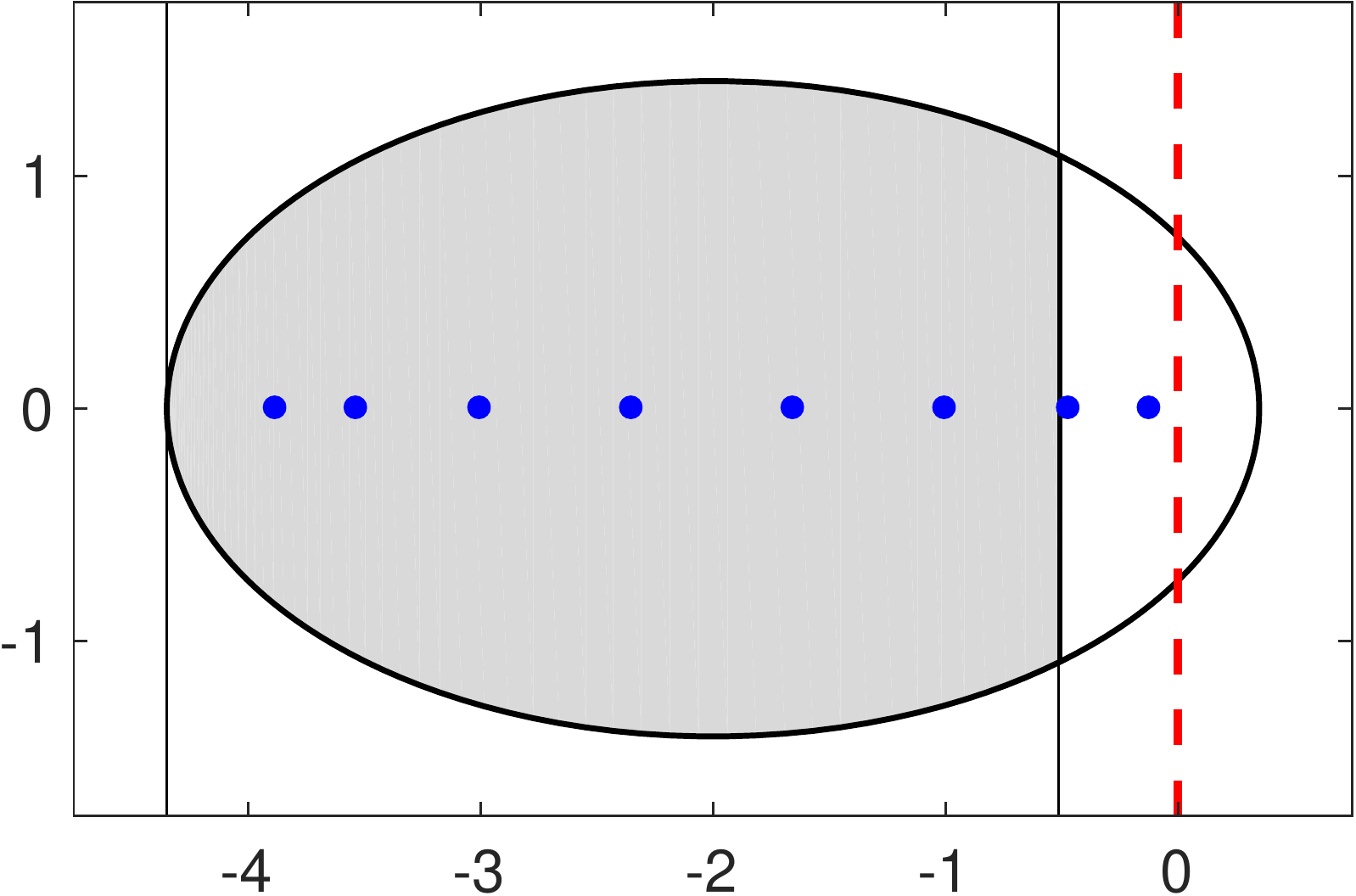}

\begin{picture}(0,0)
\put(12,95){\footnotesize $M_{-2}$}
\put(96,95){\footnotesize $M_3$}
\put(61,33){\footnotesize $\theta_3$}
\put(139,95){\footnotesize $M_{-1}$}
\put(219,95){\footnotesize $M_4$}
\put(190,33){\footnotesize $\theta_4$}
\end{picture}

\vspace*{-1em}
\caption{\label{fig:carden}
Illustration of Theorem~\ref{thm:carden} applied to the matrix~(\ref{eq:Atoep}) in
Example~\ref{ex:carden}.
The blue dots denote $\sigma(\BA)$;
the solid oval-shaped curve in the complex plane is the boundary of the numerical range $W(\BA)$; 
the dashed red line shows the imaginary axis.
The theorem ensures that the eigenvalue $\theta_j$ of $\VAV$ falls within the
associated $j$th gray subregion of $W(\BA)$.  Since only the $\theta_1$ and $\theta_2$ regions
extend into the right half-plane, $\VAV$ can have \emph{at most two unstable modes}.
}
\end{figure}

\begin{example} \label{ex:carden}
Consider the tridiagonal Toeplitz matrix
\begin{equation} \label{eq:Atoep}
 \BA = \left[\begin{array}{ccccc}
     -2 & 2 \\ 1/2 & -2 & \ddots \\ & \ddots & \ddots & 2 \\ & & 1/2 & -2
     \end{array}\right]\in\C^{8\times 8}.
\end{equation}
This matrix is stable, with negative eigenvalues%
\footnote{For eigenvalues of tridiagonal Toeplitz matrices, see~\cite[p.~59]{Smi85}.}
\[ \sigma(\BA) = \{-2+2\cos(j\pi/9): j=1,\ldots, 8\}.\]
However, $W(\BA)$ extends into the right half-plane;
indeed, 
\[ \sigma(\BH) = \{-2+{\textstyle{5\over 2}}\cos(j\pi/9): j=1,\ldots, 8\},\]
and so $\mu_1 = \omega(\BA) = 0.3492\ldots.$
Suppose we seek a ROM of dimension $k=4$.
To five digits, we compute 
\begin{center}\begin{tabular}{rr}
$M_{1} =  \phantom{-}0.34923$, & $M_{-1} = -4.34923$,\\ 
$M_{2} =  \phantom{-}0.13217$, & $M_{-2} = -4.13217$,\\ 
$M_{3} = -0.16189$, & $M_{-3} = -3.83811$,\\ 
$M_{4} = -0.51288$, & $M_{-4} = -3.48712$.\\ 
\end{tabular} \end{center}
Since only two values of $M_j$ are positive, 
Theorem~\ref{thm:carden} guarantees that \emph{no more than two eigenvalues of $\VAV$
can be in the right half-plane}.
Figure~\ref{fig:carden} illustrates $W(\BA)$ (bounded by the oval curve),
with shaded regions indicating where Theorem~\ref{thm:carden} permits $\theta_j$ to fall.
Since $M_1, M_2 > 0$, the theorem permits 
$\theta_1, \theta_2$ to be in the right half-plane;  since $M_3<0$, all other eigenvalues
$\theta_j$ for $j>2$ \emph{must} be in the left half-plane.
(Computational experiments yield rare examples where $\theta_1$ and $\theta_2=\overline{\theta}_1$ fall in the right half-plane.)
\end{example}

\begin{corollary} \label{cor:carden}
Given the notation of Theorem~\ref{thm:carden}, 
let $p$ denote the largest integer such that $M_p \ge 0$,
taking $p=0$ if $M_1< 0$.

The orthogonal projection ROM $\VAV$
can never have more than $p$ unstable modes 
(i.e., eigenvalues with nonnegative real part).
\end{corollary}

Theorem~\ref{thm:carden} and Corollary~\ref{cor:carden} need not be sharp;
they only give an upper bound on the maximal number of unstable modes.
We have let $\BA$ and $\BV$ have general complex entries.
Additional assumptions on these matrices could lead to 
sharper bounds on the eigenvalues of $\VAV$.  Most basically, 
if $\BA$ and $\BV$ have real entries, then complex eigenvalues of $\VAV$ must 
occur in conjugate pairs.

\smallskip
The following theorem uses the eigenvalues $\{\mu_j\}$ of the Hermitian part $\BA$ 
to get a lower bound on the maximal number of unstable modes.

\begin{theorem} \label{thm:lower}
Let $q$ denote the number of positive eigenvalues of $\BH = {1\over2}(\BA+\BAs^*)$,
with $0\le q \le n$.  

If $q\ge 1$, there exists a $q$-dimensional subspace of $\Cn$,
spanned by the orthonormal columns of $\BV\in\C^{n\times q}$, 
such that all eigenvalues of $\VAV$ are in the right half-plane.
\end{theorem}

\begin{proof}
The proof follows from an explicit construction.
Once again denote the eigenvalues 
of $\BH$ by $\mu_1 \ge \mu_2 \ge \cdots \ge \mu_n$; label the associated
orthonormal eigenvectors by $\Bv_1, \Bv_2, \ldots, \Bv_n$.  
Construct the projection basis via
\[ \BV := [\Bv_1\ \Bv_2\ \cdots\ \Bv_q] \in \C^{n\times q},\]
with $\BV^*\BV=\BI$.
We seek to show that $\VAV  \in \C^{q\times q}$
has $q$ positive eigenvalues.
Each $\theta \in \sigma(\VAV)$ has a unit eigenvector
$\By = [y_1\ \cdots\ y_q]^T \in\C^q$ with $\VAV \By = \theta@\By$.  
Notice that
\[ \BV\By = \sum_{j=1}^q y_j \Bv_j\]
and
\[ \By^*\VAV\By = \theta@@\By^*\By = \theta.\]
Use these expressions to compute the real part of $\theta$:
\begin{eqnarray*}
{\rm Re}\, \theta
   = {\theta + \overline{\theta} \over 2}
   &\!\!\!=\!\!\!& {1\over2}
        \By^*\BV^*(\BA+\BAs^*)\BV\By   \\
   &\!\!\!=\!\!\!& (\BV\By)^* \BH @(\BV\By) 
   = \sum_{j=1}^q \mu_j |y_j|^2,
\end{eqnarray*}
where we have used the orthonormality of the eigenvector $\Bv_j$ 
of $\BH$ for this last step.
Since $\mu_1 \ge \cdots \ge \mu_q > 0$,
\[ 
{\rm Re}\, \theta
   = \sum_{j=1}^q \mu_j |y_j|^2
    \ge  \mu_q \sum_{j=1}^q |y_j|^2
    = \mu_q > 0.\]
Thus there exists $\BV\in\C^{n\times q}$ with orthonormal columns
for which all  $q$ eigenvalues of $\VAV$ are positive.
\hfill $\qed$
\end{proof}

This theorem suggests one way to design orthogonal projection subspaces $\CV$
that yield ROMs with unstable modes, when the Hermitian part of $\BA$ has positive
eigenvalues.  In Section~\ref{sec:orth_bad}, we shall see a different approach
that takes $\CV$ to be a Krylov subspace.

\section{An upper bound on unstable modes for orthogonal projection ROMs (discrete time case)}
\label{sec:orth_disc}
A bound akin to Corollary~\ref{cor:carden} holds for 
orthogonal projection ROMs for the discrete-time system
\begin{eqnarray}
{\Bx}_{k+1} &\!\!\!=\!\!\!& \BA\Bx_k + \Bb@ u_k \label{eq:xk} \\
y_{k+1}         &\!\!\!=\!\!\!& \Bc^*\Bx_k + d@ u_k. \label{eq:yk}
\end{eqnarray}
Again let the columns of $\BV\in\Cnk$ form an orthonormal basis for
the subspace $\CV$.
Now assume that $\BA$ is stable in the discrete-time sense, i.e.,
the spectral radius $\rho(\BA)$ is less than one.
What can be said of the spectrum of $\VAV$?

In Section~\ref{sec:orth_cont}, arithmetic means of the eigenvalues 
of the Hermitian part of $\BA$ bounded the real parts of the eigenvalues of $\VAV$.
Now, \emph{geometric} means of the \emph{singular values} of $\BA$ will
bound the \emph{magnitudes} of the eigenvalues of $\VAV$.

\medskip
Let $s_1,\ldots, s_n$ denote the singular values of $\BA$.
We recall the following result from~\cite[Theorem~2.3]{CE12}.

\begin{theorem} \label{thm:carden2}
Denote the eigenvalues of $\VAV\in \Ckk$ by
$\theta_1,\ldots, \theta_k$, 
labeled by decreasing magnitude:
\[ |\theta_1| \ge |\theta_2| \ge \cdots \ge |\theta_k|.  \]
Let $G_j$ denote the geometric mean of the $j$ largest 
singular values of $\BA$,
\begin{equation} \label{eq:Gj}
  G_j := \big(s_1@\cdots s_j\big)^{1/j}, \rlap{\quad$1\le j\le n$.}
\end{equation}
Then 
\[ |\theta_j| \le G_j, \rlap{\kern32pt $1\le j\le n$.}\]
\end{theorem}

This theorem immediately gives a bound on the number of unstable modes
in a discrete-time ROM generated via orthogonal projection.

\begin{corollary} \label{cor:carden2}
Given the notation of Theorem~\ref{thm:carden2}, 
let $p$ denote the largest index for which $G_j\ge 1$, 
taking $p=0$ if $G_1 < 1$.

The orthogonal projection ROM $\VAV$ for the discrete-time 
system~$(\ref{eq:xk})$--$(\ref{eq:yk})$ 
can never have more than $p$ unstable modes
(i.e., eigenvalues with magnitude at least one).
\end{corollary}

\begin{example}
To illustrate the bounds in Theorem~\ref{thm:carden2} and Corollary~\ref{cor:carden2},
consider the stable matrix
\begin{equation} \label{eq:Adisc}
 \BA = \left[\begin{array}{ccccc}
     1/2 & \gamma \\
     1/8 & 1/2 & \gamma^2 \\
     & 1/8 & \ddots  & \ddots \\
     & & \ddots & 1/2  & \gamma^{n-1} \\
     & & & 1/8 & 1/2
   \end{array}\right]
\end{equation}
with $\gamma=3/4$ and dimension $n=128$.
This matrix is stable, with $\rho(\BA) = 0.94822\ldots,$
but the numerical range extends beyond the unit disk, with
the numerical radius $\mu(\BA) = 1.09127\ldots.$
What can be said of the stability of associated ROMs?

Label the eigenvalues of $\VAV \in \Ckk$ as 
$\theta_1,\ldots, \theta_k$, ordered by decreasing magnitude:
\[ |\theta_1| \ge |\theta_2| \ge \cdots \ge |\theta_k|.\]
For each of these eigenvalues of $\VAV$ we know
\[ \theta_j\in W(\BA) \quad \mbox{and}\quad |\theta_j| \le G_j.\]
Figure~\ref{fig:discrete} shows the regions
\[\Omega_j := W(\BA) \cap \{z\in\C: |z|\le G_j\} \]
for $j=1,\ldots,4$.
By the monotonicity of the $G_j$ values, these sets are
nested: 
\[ \Omega_k \subseteq \cdots \subseteq \Omega_{2} \subseteq \Omega_1.\]
To five digits, we compute
\begin{center}\begin{tabular}{l}
$G_{1} = 1.13227$ \\
$G_{2} = 0.99258$ \\
$G_{3} = 0.90029$ \\
$G_{4} = 0.83738$.
\end{tabular} \end{center}
Since $W(\BA)$ extends beyond the unit circle, it is possible that
$|\theta_1|>1$.  However, since $G_2 < 1$, Corollary~\ref{cor:carden2} 
ensures that \emph{all other eigenvalues $\theta_2, \ldots, \theta_k$ of $\VAV$
must be contained within the unit disk: $\VAV$ can have \emph{at most one unstable mode}.}
\end{example}

\begin{figure}
\vspace*{1.5em}
\includegraphics[scale=0.24]{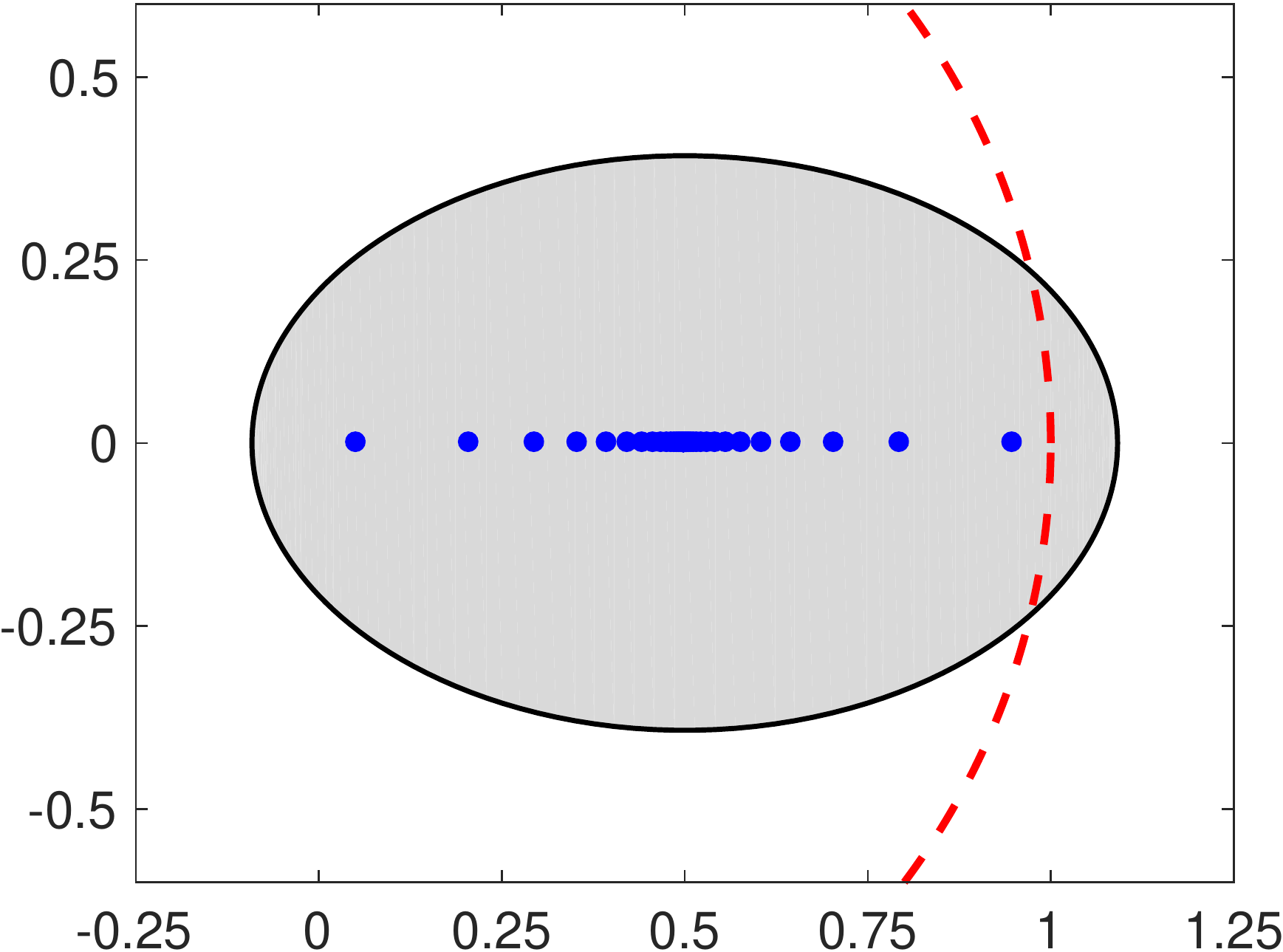}\quad
\includegraphics[scale=0.24]{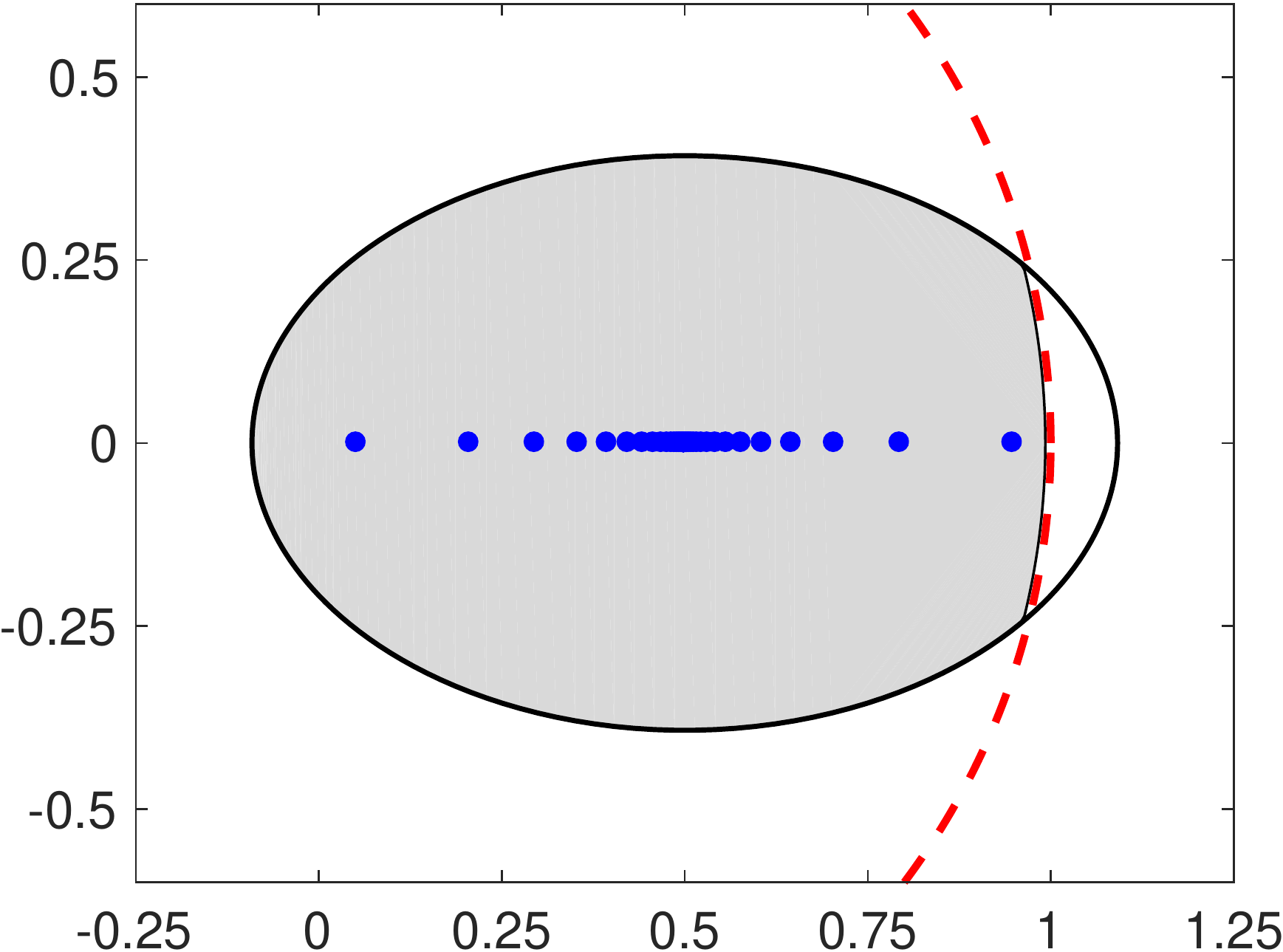}

\begin{picture}(0,0)
\put(63,40){\footnotesize $\theta_1$}
\put(195,40){\footnotesize $\theta_2$}
\end{picture}

\vspace*{1em}
\includegraphics[scale=0.24]{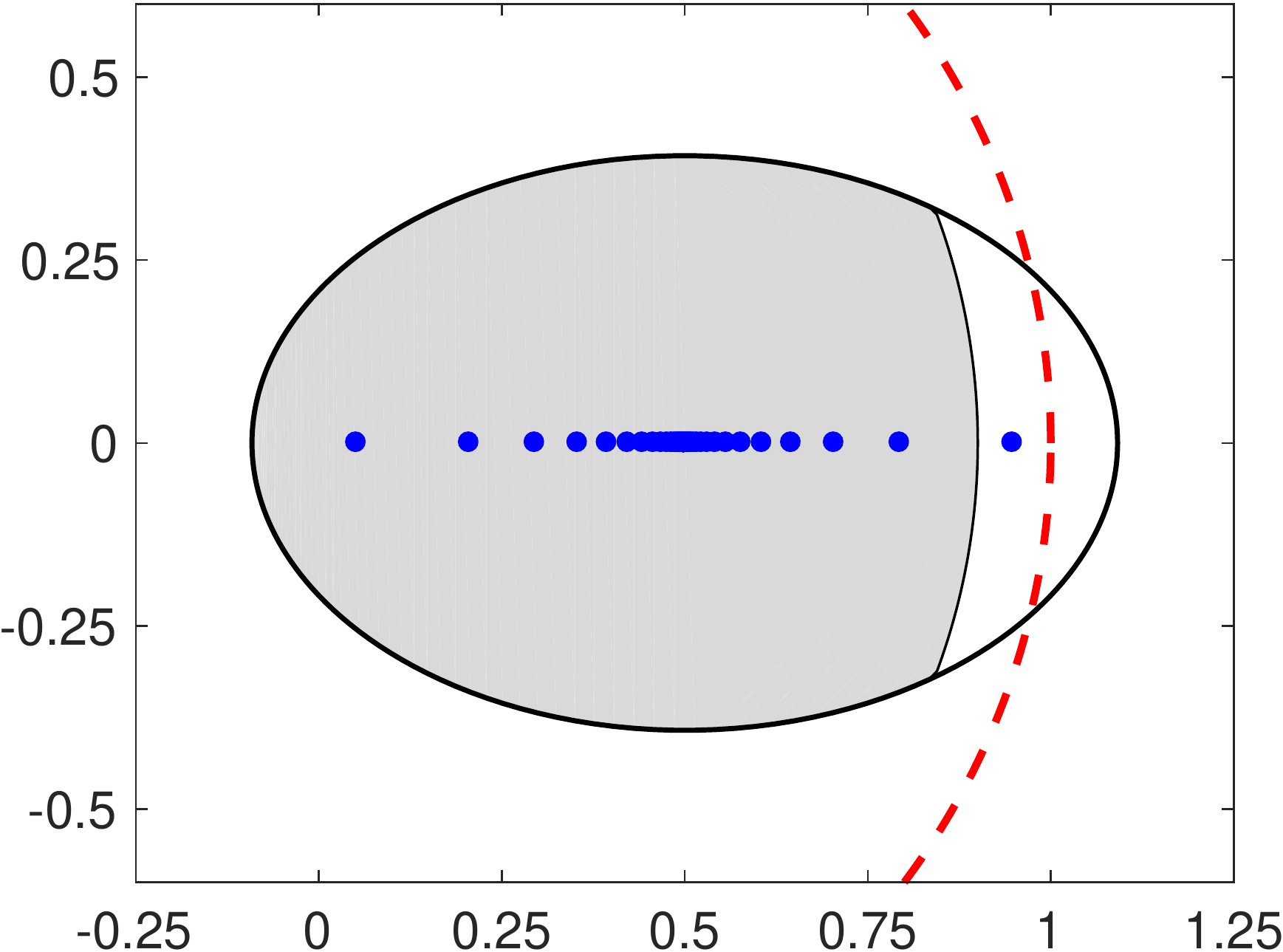}\quad
\includegraphics[scale=0.24]{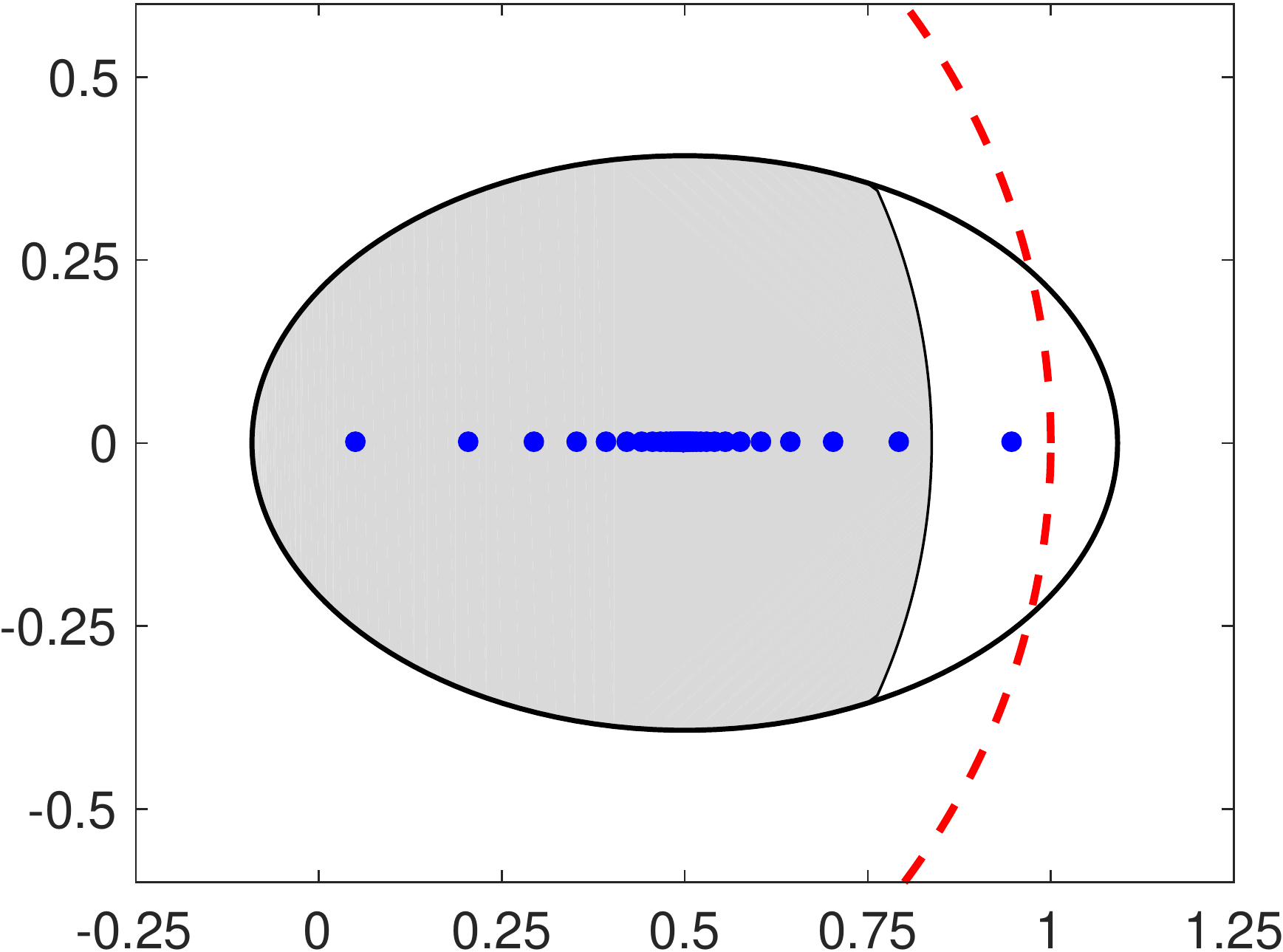}

\begin{picture}(0,0)
\put(63,40){\footnotesize $\theta_3$}
\put(195,40){\footnotesize $\theta_4$}
\end{picture}

\vspace*{-1em}
\caption{\label{fig:discrete}
Illustration of Theorem~\ref{thm:carden2} and Corollary~\ref{cor:carden2}
applied to the matrix~(\ref{eq:Adisc}).
In each plot, the black oval shows the boundary of $W(\BA)$;
the gray region shows $\Omega_j$, which must contain $\theta_j$;
the blue dots show the eigenvalues of $\BA$,
and the red dashed line shows the boundary of the unit disk. Since $\Omega_j$ is contained in
the unit disk for $j>1$, Corollary~\ref{cor:carden2} ensures $\VAV$ 
has at most one unstable mode.}
\end{figure}
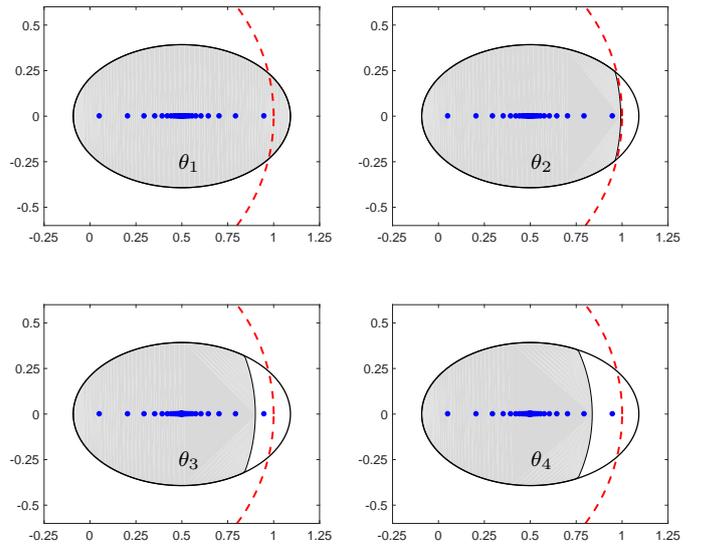


\section{Orthogonal projection: adversarial construction} \label{sec:orth_bad}

The last two sections describe how, for a given $\BA$, one 
can get rigorous limits on the number of unstable modes in a ROM
constructed using orthogonal projection from a generic subspace.
Here we describe a construction for probing extreme 
limits of instability, provided one is content to only fix the eigenvalues of $\BA$
but let the departure from normality vary.
This result was proved by Duintjer Tebbens and Meurant~\cite[Corollary~2.3]{DM12}, 
a contribution to the convergence theory for Arnoldi's algorithm for 
computing eigenvalues; it builds on earlier work of 
Greenbaum, Pt\'ak, and Strakos~\cite{GS94,GPS96}.

\begin{theorem} \label{thm:dtm}
Let $\Sigma = \{\lambda_1, \ldots, \lambda_n\}\subset \C$ 
denote a collection of desired eigenvalues, and specify any values for
\begin{eqnarray*}
\Sigma_1 \!\!\!&:=&\!\!\! \{\theta_1^{(1)}\};\\[.5em]
\Sigma_2 \!\!\!&:=&\!\!\! \{\theta_1^{(2)}, \theta_2^{(2)}\};\\[.5em]
\Sigma_3 \!\!\!&:=&\!\!\! \{\theta_1^{(3)}, \theta_2^{(3)}, \theta_3^{(3)}\};\\[.25em]
\quad \!\!\!&\vdots&\!\!\! \\[.25em]
\Sigma_{n-1} \!\!\!&:=&\!\!\! \{\theta_1^{(n-1)}, \theta_2^{(n-1)}, \ldots, \theta_{n-1}^{(n-1)}\}.
\end{eqnarray*}
There exists a matrix $\BA$ and a vector $\Bb$ such that
\[ \sigma(\BA) = \Sigma = \{\lambda_1,\ldots, \lambda_n\}\]
and, for $k=1,\ldots, n-1$,
\[ \sigma(\VkAVk) = \Sigma_k = \{\theta_1^{(k)}, \ldots, \theta_{k}^{(k)}\},\]
where the columns of $\BVs_k$ form an orthonormal basis for 
the Krylov subspace ${\rm range}(\BVs_k) = \CK_k(\BA,\Bb)$.
\end{theorem}

We can use this theorem (and its constructive proof) to build stable $\BA$
and corresponding $\Bb$ for which \emph{all eigenvalues of the orthogonal 
projection ROMs $\VkAVk$ from the Krylov subspace $\CK_k(\BA,\Bb)$ fall at
any desired location in the right half-plane}, for $k=1,\ldots, n-1$, despite
the fact that these ROMs all match $k$ moments.
The next example illustrates this point.

\begin{figure}[b!]
\begin{center}
\includegraphics[scale=0.55]{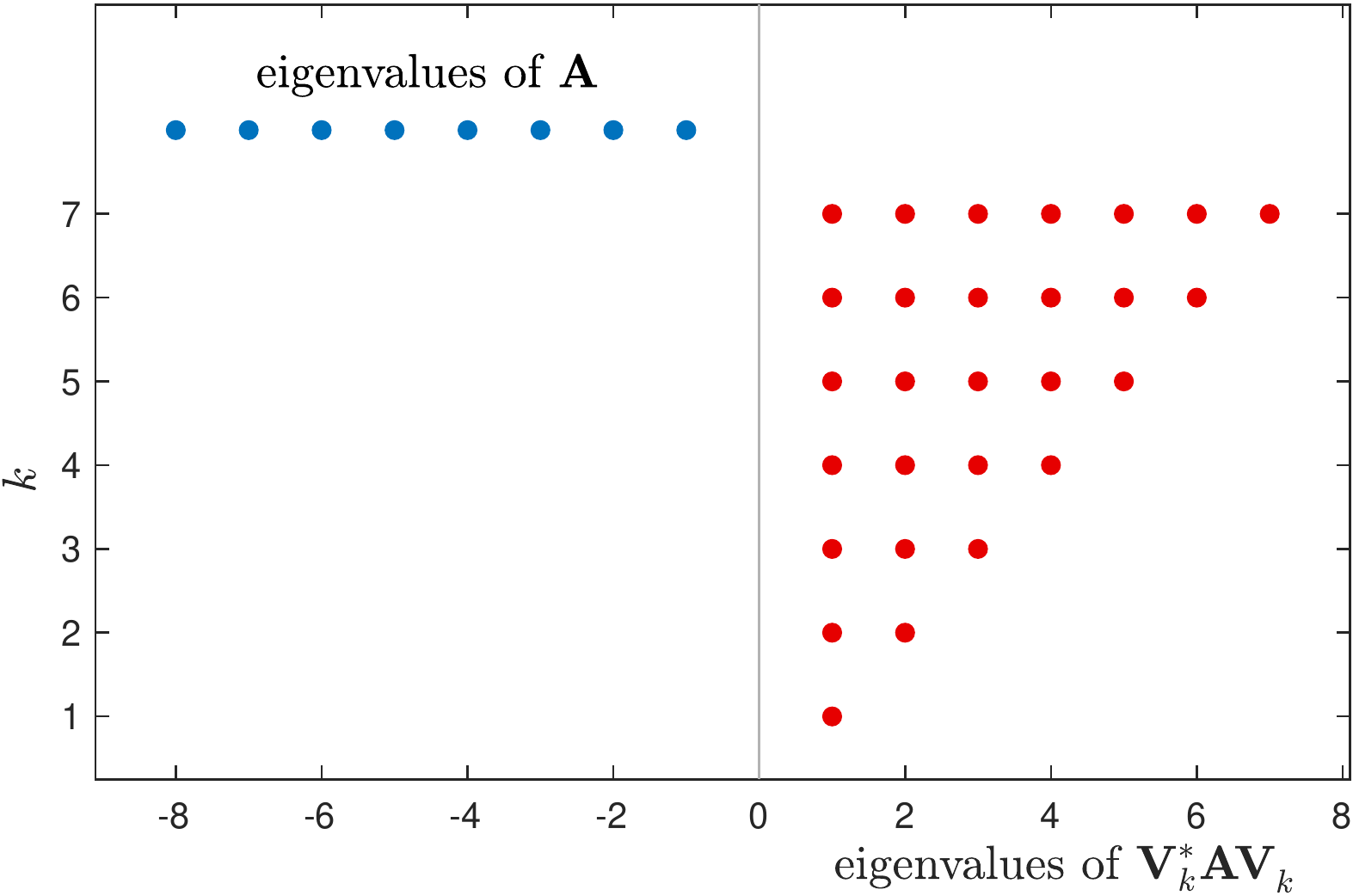}
\end{center}

\vspace*{-10pt}
\caption{ \label{fig:DM12ex}
Comparison of the eigenvalues of $\BA$ (blue dots, in the left half-plane)
to those of $\VkAVk$ of order $k=1,\ldots, 7$ for the system in Example~\ref{ex:dm12},
generated using orthogonal projection onto Krylov subspaces.
These ROMs are entirely unstable, even through 
the $k$th order ROM matches $k$ moments of the original system.
}
\end{figure}

\begin{example} \label{ex:dm12}
Using the construction described by Duintjer Tebbens and Meurant~\cite[Proposition~2.1]{DM12},
we form 
\[ \BA = \left[\!\begin{array}{cccccccr}
1 & 0 & 0 & 0 & 0 & 0 & 0\!\! &  -362880    \\
1 & 2 & 0 & 0 & 0 & 0 & 0\!\! &  -1451520    \\
  & 1 & 3 & 0 & 0 & 0 & 0\!\! &  -1693440    \\
  &   & 1 & 4 & 0 & 0 & 0\!\! &  -846720  \\
  &   &   & 1 & 5 & 0 & 0\!\! &  -211680\\
  &   &   &   & 1 & 6 & 0\!\! &  -28224 \\
  &   &   &   &   & 1 & 7\!\! &  -2016 \\
  &   &   &   &   &   & 1\!\! &  -64 \\
\end{array}\!\right],
\ 
 \Bb = \left[\!\begin{array}{c}
 1 \\ 0 \\ 0 \\ 0 \\ 0 \\ 0 \\ 0 \\ 0
\end{array}\!\right].
\]
The matrix $\BA$ was constructed to have the stable spectrum
\[ \sigma(\BA) = \{-1, -2, -3, -4, -5, -6, -7, -8\}.\]
The form of $\BA$ and $\Bb$ makes it easy to write down the associated Krylov subspace,
\[ {\rm range}(\BVs_k) = \CK_k(\BA,\Bb) = {\rm span}\{\Be_1, \ldots, \Be_k\},\]
so that $\VkAVk$ is the $k\times k$ principal submatrix of $\BA$.
Since these submatrices are lower triangular, one can easily read off their eigenvalues:
\[ \sigma(\VkAVk) = \{ 1,\ldots, k\}.\]
Though the $k$th order ROM matches $k$ moments of the stable system,
all $k$ modes are unstable.  
Figure~\ref{fig:DM12ex} contrasts these unstable modes with the 
stable eigenvalues of $\BAs$.

\begin{figure}[b!]
\begin{center}
\includegraphics[scale=0.45]{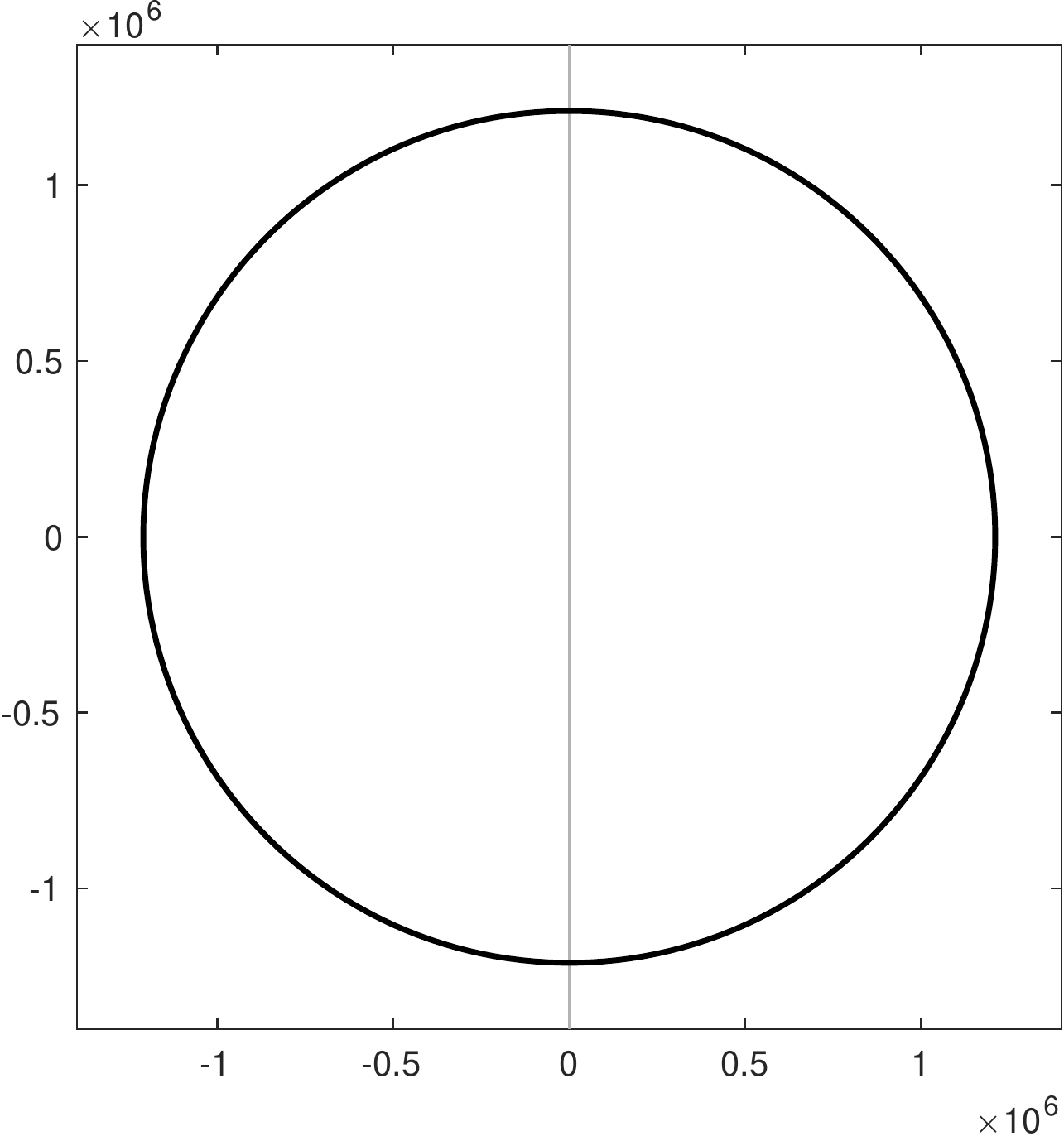}

\begin{picture}(0,0)
\put(50,40){$W(\BA)$}
\end{picture}

\includegraphics[scale=0.5]{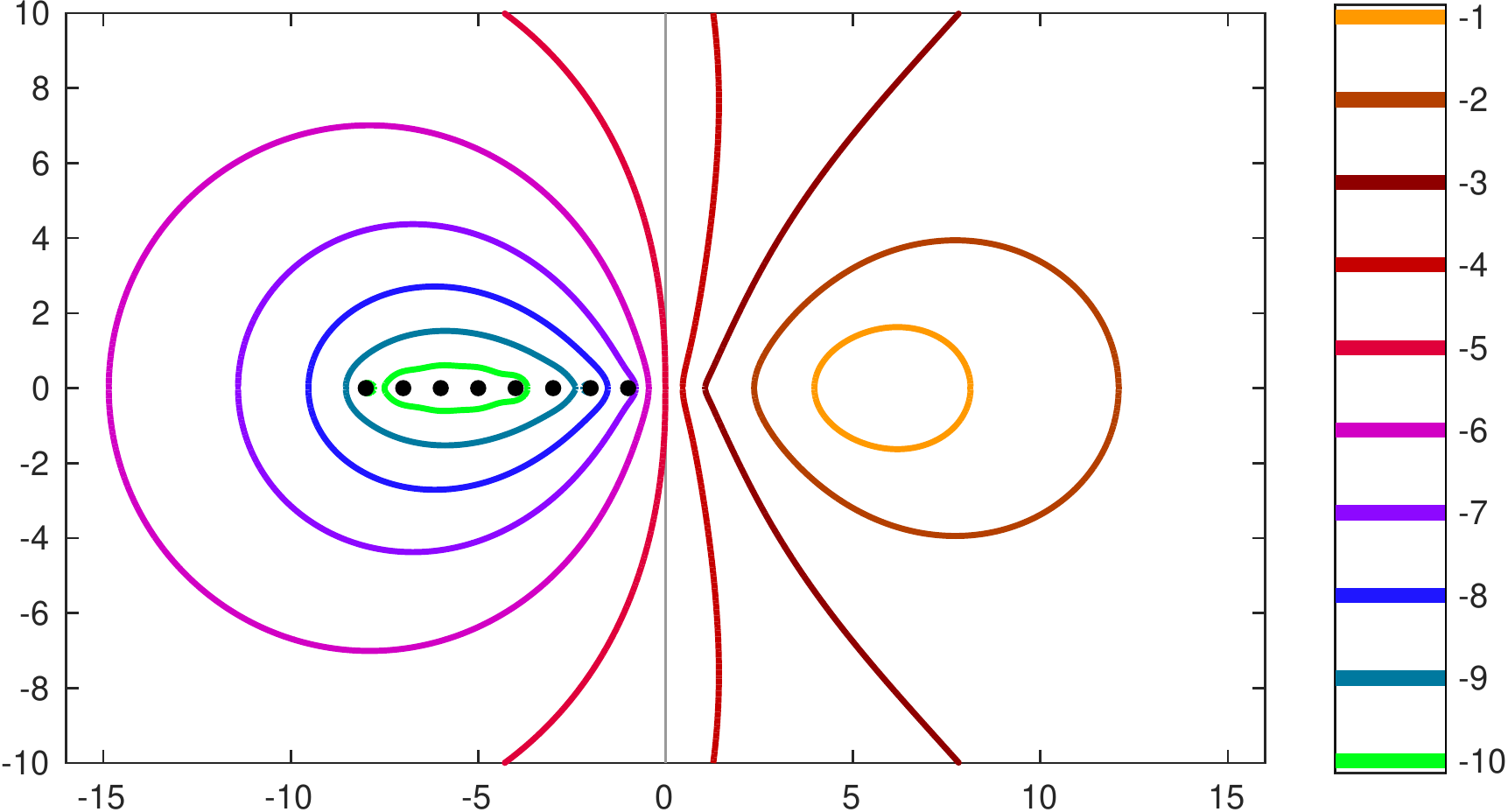}

\begin{picture}(0,0)
\put(50, 30){$\sigma_\eps(\BA)$}
\end{picture}
\end{center}

\vspace*{-22pt}
\caption{ \label{fig:DM12ex2}
The numerical range $W(\BA)$ (top) and $\eps$-pseudospectra $\sigma_\eps(\BA)$ (bottom)
for the stable $\BA$ from Example~\ref{ex:dm12}.
Since $W(\BA)$ extends far into the right half-plane, solutions to 
$\dot\Bx(t)=\BA\Bx(t)$ can exhibit significant transient growth before
asymptotic decay.  
In the bottom plot (computed using EigTool~\cite{Wri02a}), 
the lines show the boundaries of $\sigma_\eps(\BA)$, with colors 
corresponding to $\log_{10}(\eps)$. 
For $\eps=10^{-4}$, $\sigma_\eps(\BA)$ clearly extends into the right half-plane:
there exist nearby matrices $\BA+\BE$ that are unstable, with $\|\BE\| < 10^{-4}$.
}
\end{figure}

This construction can deliver such startling results because 
\emph{we only specify the eigenvalues of $\BA$}.
One might suspect that the $\BA$ this construction produces might
have a significant departure from normality, corresponding 
to transient growth of the dynamical system and eigenvalue instability.
Indeed, Figure~\ref{fig:DM12ex2} confirms this departure from normality.
The numerical range $W(\BA)$ extends beyond $10^6$ into the right half-plane
($\omega(\BA) \approx 1.211 \times 10^6$), signaling rapid growth 
of $\|\eop^{t\BA}\|$ for small $t$; see~(\ref{eq:t0}).
The $\eps$-pseudospectra reveal that $\BA$ is close to an unstable system:
since $\sigma_\eps(\BA)$ extends into the right half-plane for 
$\eps=10^{-4}$, there exist matrices $\BE\in\C^{8\times 8}$ with
$\|\BE\|<10^{-4}$ that make $\BA+\BE$ unstable.

In short, this pathological example corresponds to a special
$\BA$ with unusual dynamics and a fragile spectrum.
(Our $\Sigma$ and $\Sigma_k$ are inspired by an ill-conditioned 
pole placement example of Mehrmann and Xu~\cite[Example~2]{MX96}.)
\end{example}

\section{Oblique projection: adversarial construction} \label{sec:oblique}

Like the orthogonal projection methods addressed in the previous sections,
oblique projection methods approximate the true state vector $\Bx(t) \approx \BV\xhat(t) \in \CV$,
for some $k$-dimensional subspace ${\rm range}(\BV) = \CV$.  
Oblique projection methods replace the orthogonality constraint~(\ref{eq:Galerkin})
with the \emph{Petrov--Galerkin condition}
\begin{equation} \label{eq:PG}
 \BW^*\Big(\BV@\dot{\xhat}(t) - \big(\BA\BV@\xhat(t) + \Bb @u(t)\big)\Big) = \Bzero,
\end{equation}
where ${\rm range}(\BW) = \CW$ is some (generally different) $k$-dimensional subspace.
The bases for $\CV$ and $\CW$ stored in the columns of $\BV$ and $\BW$ are now 
constructed to be \emph{biorthogonal}: $\BW^*\BV = \BI$.\ \ 
The Petrov--Galerkin constraint~(\ref{eq:PG}) gives the reduced system
\begin{equation} \label{eq:oblique}
 \dot{\xhat}(t) = (\WAV) @@\xhat(t) + (\BW^*\Bb) @u(t).
\end{equation}

The balanced truncation method
(see, e.g., \cite[Chapter~7]{Ant05b}, \cite[Chapter~7]{Zho95})
fits this template, and generates models that are guaranteed to preserve stability.
However, to compute the balancing biorthogonal bases $\BV$ and $\BW$ 
one  must solve two Lyapunov matrix equations 
(typically at considerable computation expense, though algorithmic improvements
make this increasingly tractable for large-scale problems; see, e.g., \cite{BS13,Pen00a,Sab06,Sim16}).

The bi-Lanczos algorithm provides an inexpensive alternative for constructing
oblique projection models.  
This method has the advantage that the biorthogonal bases can be computed using 
three-term vector recurrences that only require multiplication by $\BA$ and $\BAs^*$ 
to construct biorthogonal bases for $\CK_k(\BA,\Bb)$ and $\CK_k(\BAs^*,\Bc)$ 
(that form the columns of $\BV$ and $\BW$).

The resulting order-$k$ bi-Lanczos ROM \emph{matches $2k$ moments}
of the transfer function, while orthogonal projection onto the Krylov subspace 
$\CK_k(\BA,\Bb)$ produces a ROM that only matches $k$ moments~\cite[Section~11.2]{Ant05b}.
Moreover, the three-term recurrence behind bi-Lanczos makes the bases $\BV$ and $\BW$
quicker to compute than the orthonormal basis for $\CV$ required by the
orthogonal projection method (which uses long recurrences in the Gram--Schmidt process).

Despite these advantages, the bi-Lanczos method often
suffers from significant numerical instability.
While the bases that form the columns of $\BV$ and $\BW$ are biorthogonal,
the columns of these two matrices might themselves be quite ill-conditioned
bases for $\CK_k(\BA,\Bb)$ and $\CK_k(\BAs^*,\Bc)$.
In extreme cases, the method can \emph{break down}.
More often the iterations come close to failure, 
exhibiting numerical instabilities; see, e.g., \cite{PTL85}. 
The problem is apparent even when $k=1$.
Suppose that $\Bc^*\Bb=0$,
as could easily occur in a physical system where the input occurs at a point far 
from the output measurement.
In this case there exists no biorthogonal bases for 
$\CV = \CK_1(\BA,\Bb) = {\rm span}(\Bb)$ and 
$\CW = \CK_1(\BAs^*,\Bc) = {\rm span}(\Bc)$.

When the bi-Lanczos procedure succeeds without breakdown, 
it produces the factorizations
\begin{subequations}
\begin{eqnarray} 
    \BA\BVs_k^{} \!\!\!&=& \!\!\!\BVs_k^{} \BT_k^{} + \gamma_k^{} \Bv_{k+1}^{} \Be_k^* \label{eq:bilancV} \\
    \BAs^*\BWs_k^{} \!\!\!&=& \!\!\!\BWs_k^{} \BT_k^* + \overline{\beta_k^{}} \Bw_{k+1}^{} \Be_k^*; \label{eq:bilancW}
\end{eqnarray}
\end{subequations}
premultiplying~(\ref{eq:bilancV}) by $\BW_k^*$ and using the biorthogonality of 
the basis yields 
\[ \WkAVk = \BT_k
    = 
   \left[\!\!\begin{array}{cccc}
     \alpha_1 & \beta_1 \\
     \gamma_1 & \alpha_2 & \ddots \\
              & \ddots & \ddots & \beta_{k-1} \\
               & & \gamma_{k-1} & \alpha_k 
    \end{array}\!\!\right] \in \C^{k\times k}.\]
The \emph{tridiagonal} structure is inherited from the three-term
recurrence relations at the heart of the bi-Lanczos process.
Indeed, the $j$th columns of equations~(\ref{eq:bilancV})--(\ref{eq:bilancW}) give
\begin{subequations} \label{eq:biLanc}
\begin{eqnarray}
  \gamma_j \Bv_{j+1}\!\!\! &=&\!\!\! \BA\Bv_j - \alpha_j \Bv_j - \beta_{j-1} \Bv_{j-1}, \label{eq:biLancv}\\
  \overline{\beta_j} \Bw_{j+1}\!\!\! &=&\!\!\! \BAs^*\Bw_j - \overline{\alpha_j} \Bw_j - \gamma_{j-1} \Bw_{j-1}
\end{eqnarray}
\end{subequations}
for $j=1,\ldots, k$, with $\beta_0 = \gamma_0 =0$ and $\Bv_0 = \Bw_0 = \Bzero$.
(We assume $\gamma_j\in \R$, a natural choice in bi-Lanczos codes.)

\subsection{Greenbaum's theorem}

Few concrete results are known about the spectra of ROMs generated
using the bi-Lanczos process.
The most substantial insight comes from Greenbaum~\cite[Theorem~3]{Gre98}
(motivated by the study of bi-Lanczos-based iterative methods for solving $\BA\Bx=\Bb$).
We will interpret this result in the context of moment-matching model reduction.

\medskip
Suppose we are given a system of order~$n$, 
\[ \dot\Bx(t) = \BA\Bx(t) + \Bb u(t),\]
for which we seek a ROM of order $k\le n/2$.
(The output vector $\Bc$ will be constructed later.)

Choose \emph{any} parameters 
\begin{equation} 
 \alpha_1, \ldots, \alpha_k, \label{eq:alpha} 
\end{equation}
\begin{equation} 
 \beta_1, \ldots, \beta_{k-1}, \label{eq:beta}
\end{equation}
and construct $\gamma_1,\ldots, \gamma_{k-1}$ by running the recurrence
\begin{subequations} \label{eq:3term}
\begin{eqnarray}
\widehat{\Bv}_{j+1} \!\!\!&:=&\!\!\! \BA\Bv_j - \alpha_j \Bv_j - \beta_{j-1} \Bv_{j-1} \label{eq:3termA} \\[.1em]
          \gamma_j  \!\!\!&:=& \!\!\!\|\widehat{\Bv}_{j+1}\| \label{eq:gamma}\\[.1em]
          \Bv_{j+1} \!\!\! &:=& \!\!\!\widehat{\Bv}_{j+1} / \gamma_j, 
\end{eqnarray}
\end{subequations}
for $j=1,\ldots, k$, with $\Bv_1 = \Bb/\|\Bb\|$, $\Bv_0 = \Bzero$, and $\beta_0=0$;
this recurrence is obviously meant to mimic~(\ref{eq:biLancv}).

\begin{theorem} \label{thm:greenbaum}
Consider the parameters $\{\alpha_j\}_{j=1}^k$, $\{\beta_j\}_{j=1}^{k-1}$, and
$\{\gamma_j\}_{j=1}^{k}$ from~$(\ref{eq:alpha})$, $(\ref{eq:beta})$, and {\rm (\ref{eq:gamma})},
and the vectors $\Bv_1,\ldots, \Bv_{k+1}$ from~$(\ref{eq:3term})$,
with $k\le n/2$.
Suppose $\Bc\in\Cn$ satisfies
\begin{equation} \label{eq:obliquec}
  \Bc \perp {\rm span}\{\Bv_2, \ldots, \Bv_{k+1}, \BA\Bv_{k+1}, \ldots, \BAs^{k-1}\Bv_{k+1}\}.
\end{equation}
Then either the bi-Lanczos process breaks down, or it runs to completion and creates the
ROM
\[  \WkAVk
    =
   \left[\!\!\begin{array}{cccc}
     \alpha_1 & \beta_1 \\
     \gamma_1 & \alpha_2 & \ddots \\
              & \ddots & \ddots & \beta_{k-1} \\
               & & \gamma_{k-1} & \alpha_k
    \end{array}\!\!\right] \in \C^{k\times k}.\]
\end{theorem}
The cases of breakdown in Theorem~\ref{thm:greenbaum} correspond, for example,
to scenarios where $\beta_j=0$ or $\gamma_j=0$, since these values are used
to normalize $\Bw_{j+1}$ and $\Bv_{j+1}$ in the Lanczos algorithm.
(We shall not concern ourselves with look-ahead procedures (see, e.g.,~\cite{PTL85}) here,
which can provide a work-around to many instances of breakdown.)

\smallskip
As Greenbaum evocatively describes it, Theorem~\ref{thm:greenbaum} essentially says
that the bi-Lanczos process can be viewed as executing an arbitrary recurrence 
for the first $n/2$ steps, provided one is free to select an appropriate left-starting
vector $\Bc$. 
Since these recurrence coefficients are essential elements of the resulting ROM,
Theorem~\ref{thm:greenbaum} suggests a way to design cases where benign choices of $\BA$
(even stable, normal or Hermitian matrices)
lead to unstable ROMs.
The next example gives an extreme illustration.

\begin{figure}[b!]
\begin{center}
\includegraphics[scale=0.55]{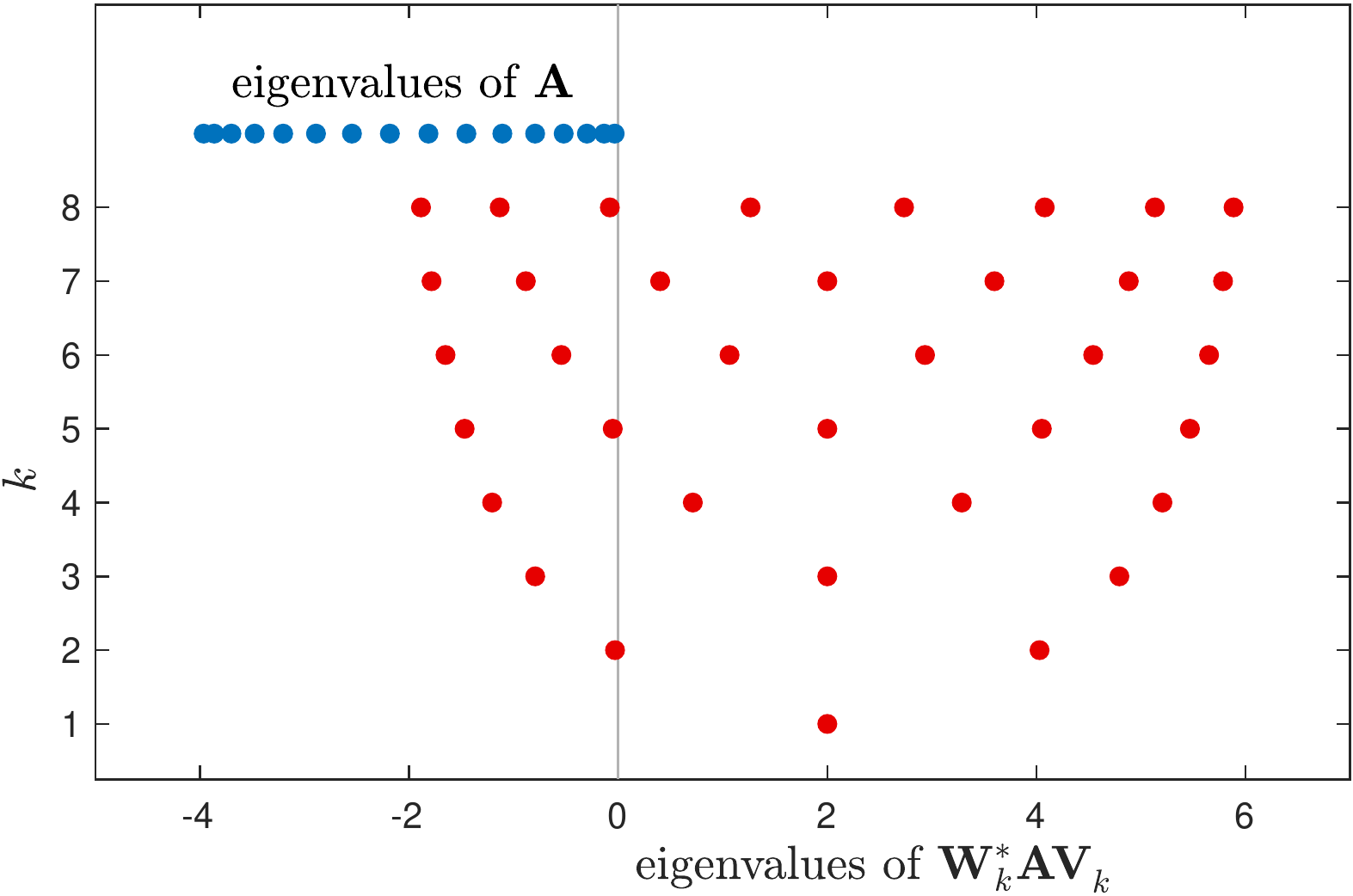}
\end{center}

\vspace*{-14pt}
\caption{\label{fig:greenbaum}
Comparison of the eigenvalues of $\BA$ (blue dots) in the left half-plane 
to those of the ROMs of order $k=1,\ldots, 8$ from 
Example~\ref{ex:greenbaum}, generated using the bi-Lanczos algorithm that
matches $2k$ moments.
Despite the fact that $\BA$ is stable and Hermitian, 
all of these ROMs are highly unstable.
}
\vspace*{-14pt}
\end{figure}

\begin{example} \label{ex:greenbaum}
Consider a stable continuous-time system with Hermitian matrix
\[ \BA = \left[\begin{array}{cccc}
      -2 & 1 \\ 1 & -2 & \ddots \\ & \ddots & \ddots & 1 \\ & & 1 & -2
      \end{array}\right] \in \C^{16\times 16}, 
\]
and input vector $\Bb = [1, 0, \ldots, 0]^T \in \C^{16}$.
The eigenvalues of $\BA$ are all negative real numbers:
\[ \sigma(\BA) = \bigg\{ -2+ 2 \cos\Big({k\pi \over 17}\Big): k=1,\ldots, 16\bigg\}.\]
Since $\BA$ is Hermitian, $W(\BA)$ is the convex hull of the spectrum, 
and \emph{any orthogonal projection method must produce a stable ROM.}\ \ 
Greenbaum's theorem shows that oblique projection methods can produce 
much more exotic results.

Suppose we seek a ROM of order $k=8$, and specify the bi-Lanczos recurrence
parameters
\begin{eqnarray*}
    \alpha_1 = \cdots = \alpha_8 \!\!\! &=&\!\!\! 2,  \\[.25em]
    \beta_1= \cdots = \beta_7 \!\!\! &=&\!\!\! 1. 
\end{eqnarray*}
(These parameters were selected to give a reduction that was likely to be unstable.)
From the three-term recurrence~(\ref{eq:3term}) we compute (to five digits)
\begin{center}\begin{tabular}{ll}
    $\gamma_1 = 4.12311$, & $\gamma_2 = 3.68474$, \\[.15em]
    $\gamma_3 = 4.12603$, & $\gamma_4 = 4.31536$, \\[.15em]
    $\gamma_5 = 4.43571$, & $\gamma_6 = 4.52257$, \\[.15em]
    $\gamma_7 = 4.58628$. 
\end{tabular}\end{center}
The vector $\Bc$ is then constructed, consistent with Theorem~\ref{thm:greenbaum},
by orthogonally projecting $\BA\Bb$ onto the orthogonal complement of the span
in~(\ref{eq:obliquec}).

Figure~\ref{fig:greenbaum} shows the eigenvalues of the matrix $\WkAVk$ from
the resulting ROM for $k=1,\ldots, 8$.
Despite the fact that $\BA$ is Hermitian and stable, each of the models of
order $k=1,\ldots, 8$ is \emph{unstable}; e.g., the ROMs of order $k=7$ and $k=8$ both 
have five unstable modes.
\end{example}

In closing this section, we emphasize a fundamental distinction between the
adversarial construction 
in Theorem~\ref{thm:dtm} for orthogonal projection methods,  
and its counterpart 
in Theorem~\ref{thm:greenbaum} for oblique
projection. 
In the orthogonal case, \emph{the user can only specify the eigenvalues of $\BA$};
the construction obtains pathological results by building $\BA$ with large departure 
from normality, along with a corresponding $\Bb$.
In the oblique case, \emph{the user supplies the entire matrix $\BA$
and input vector $\Bb$}.
The construction in no way influences the departure of $\BA$ from normality;
it achieves its ends only by designing the output vector $\Bc$.
Thus the oblique construction is more troubling, as it can produce 
startling results for apparently benign $\BA$.

\section{The unsung merits of unstable ROMs} \label{sec:unstable}

Conventionally, an unstable ROM seems to be a poor approximation 
to a stable dynamical system, regardless of its other virtues 
(e.g., transfer functions that match moments).
Missing the asymptotic character of the model is a fundamental shortcoming.
Before dismissing unstable ROMs entirely, one should note that they can
still give insight into the original system, especially regarding 
transient dynamics.  We warn that efforts to suppress the instability can 
have the unintended consequence of compromising transient accuracy to 
obtain long-term qualitative agreement.
We illustrate these points with two examples.

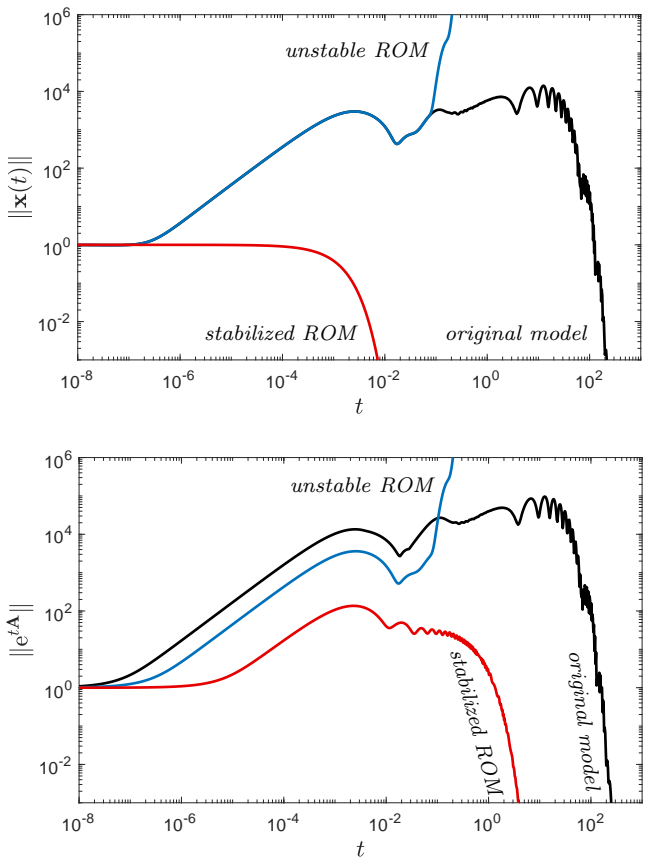
\begin{figure}[b!]
\begin{center}
\includegraphics[scale=.5]{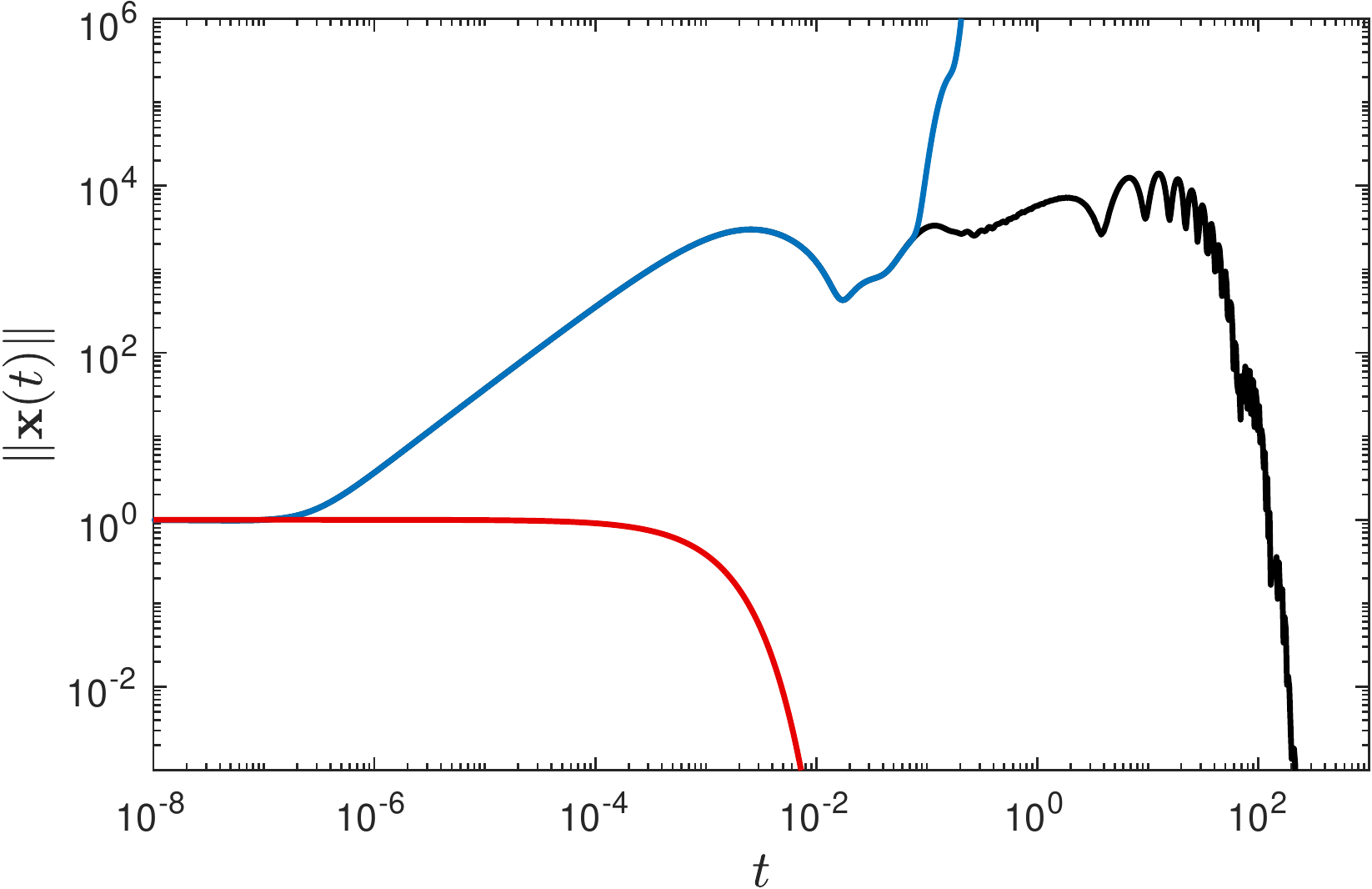}

\begin{picture}(0,0)
\put(-15,147){\footnotesize \emph{unstable ROM}}
\put(-45,40){\footnotesize \emph{stabilized ROM}}
\put(45,40){\footnotesize \emph{original model}}
\end{picture}

\includegraphics[scale=.5]{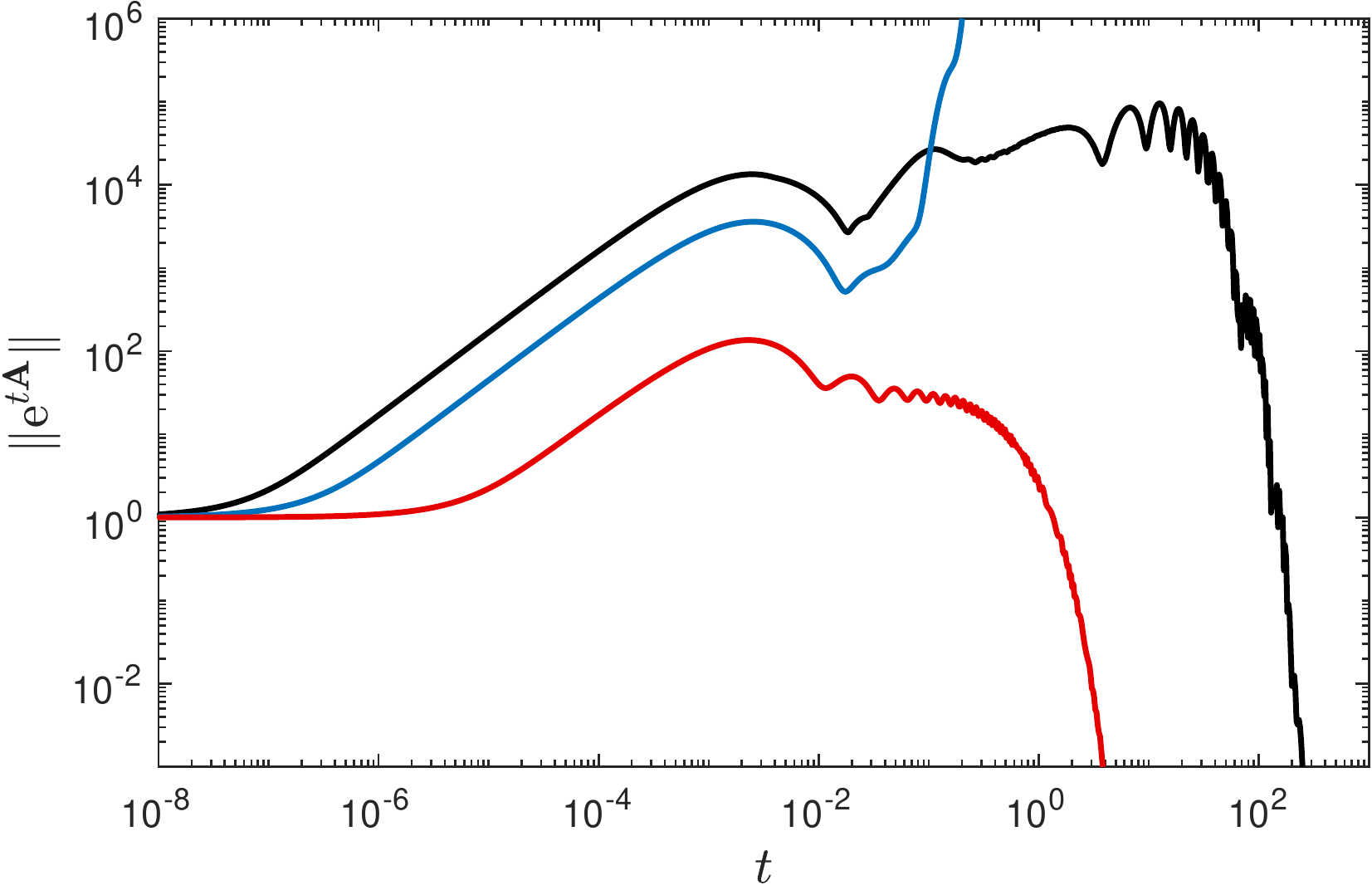}

\begin{picture}(0,0)
\put(-13,150){\footnotesize \emph{unstable ROM}}
\put(90,90){\rotatebox{-83}{\footnotesize \emph{original model}}}
\put(47,90){\rotatebox{-75}{\footnotesize \emph{stabilized ROM}}}
\end{picture}
\end{center}

\vspace*{-28pt}
\caption{\label{fig:unstable1}
Evolution of a solution $\Bx(t)$ to $\dot \Bx(t) = \BA\Bx(t)$ for the 
Boeing B-767 matrix in Example~\ref{ex:boeing} (top), 
and the analogous plot for the solution operator $\eop^{t\BA}$ (bottom).
The unstable ROM provides a much more accurate impression of 
the dynamics at early $t$ than does the stabilized model 
(both of order $k=20$).
}
\end{figure}

\begin{figure}[t!]
\begin{center}
\includegraphics[scale=.5]{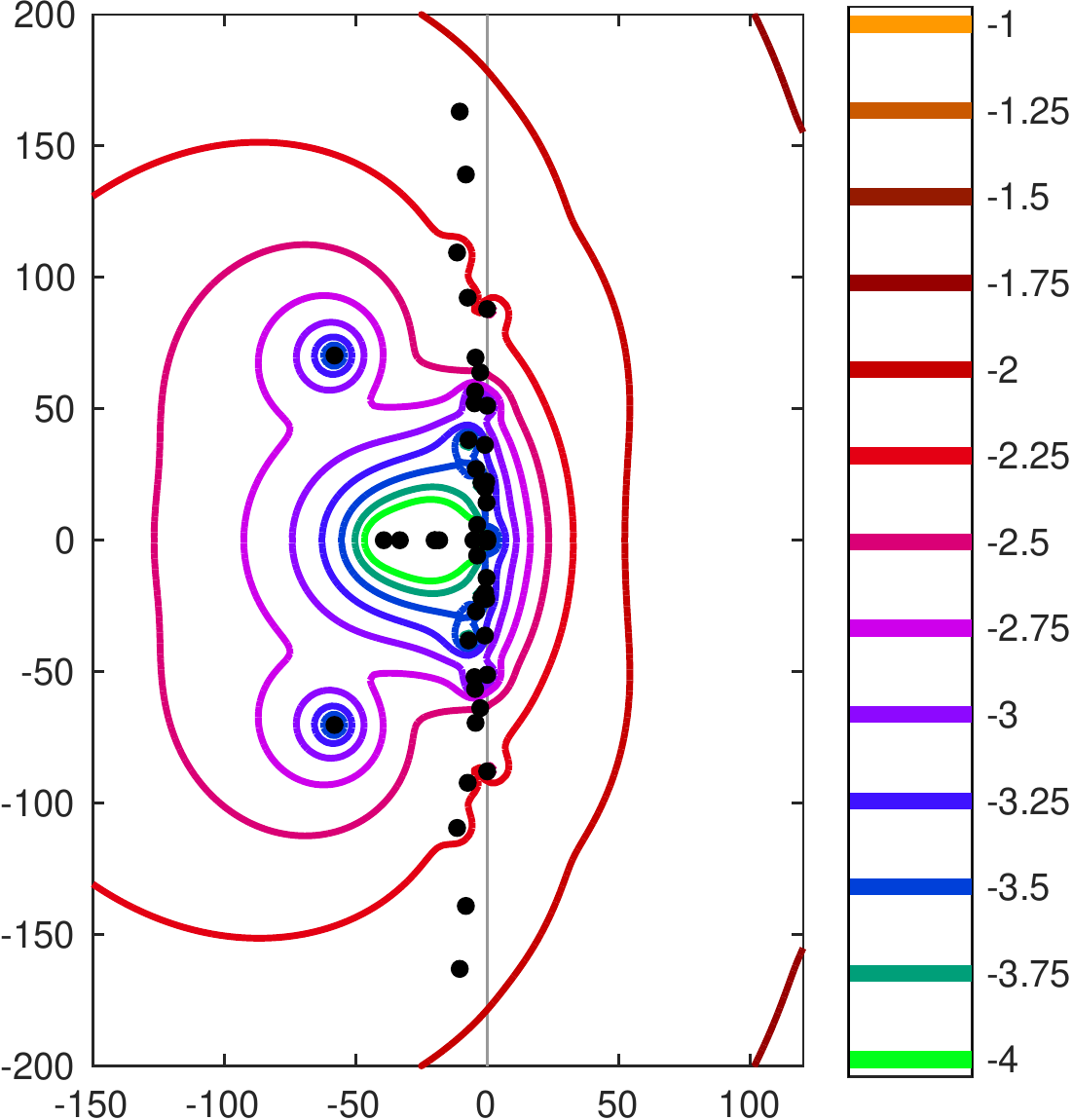}

\begin{picture}(0,0)
\put(-30,0){\footnotesize \emph{original model}}
\end{picture}

\vspace*{12pt}
\includegraphics[scale=.5]{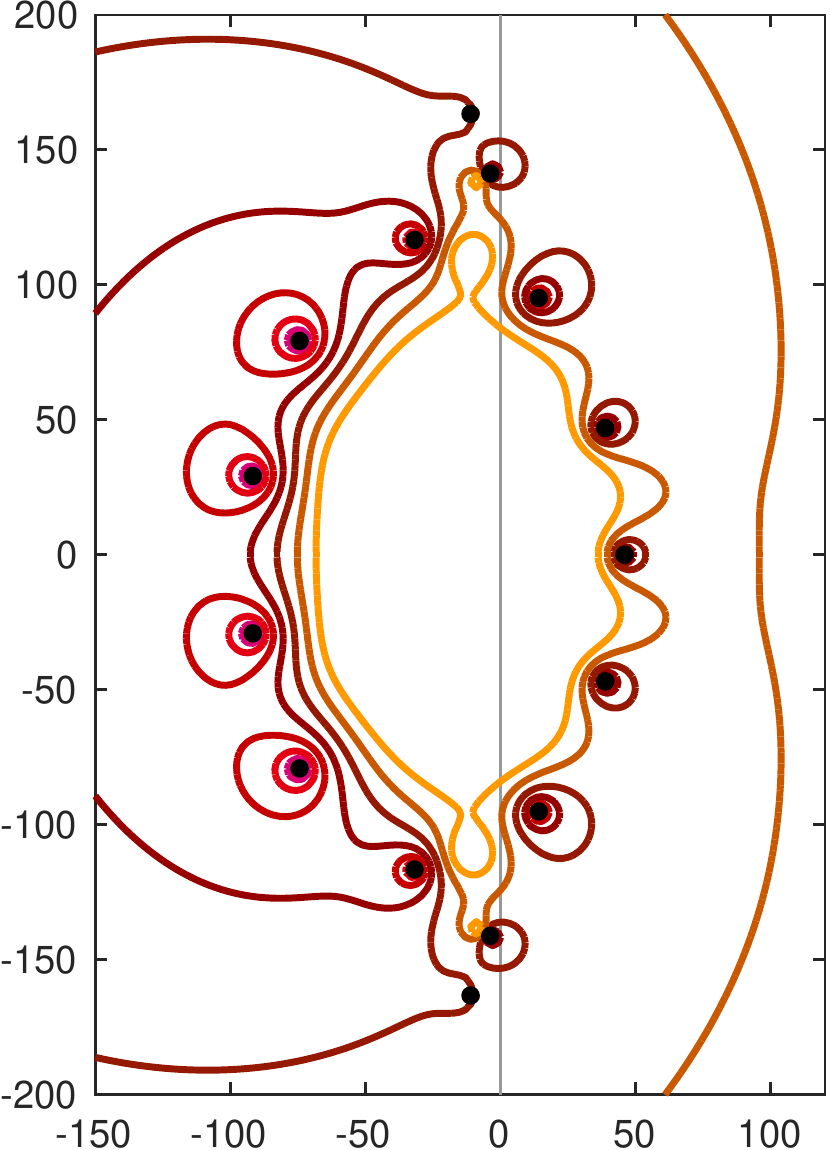}\quad
\includegraphics[scale=.5]{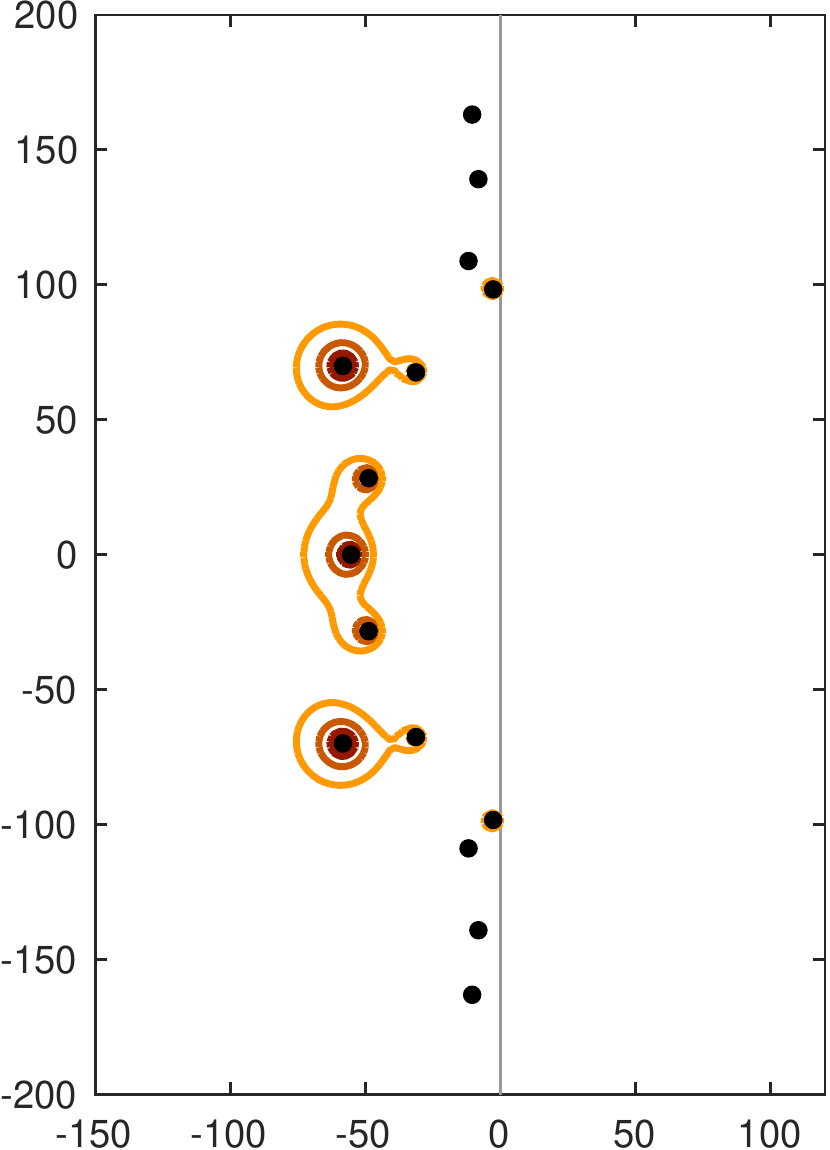}

\begin{picture}(0,0)
\put(-85,0){\footnotesize \emph{unstable ROM}}
\put(40,0){\footnotesize \emph{stabilized ROM}}
\end{picture}
\end{center}

\vspace*{-10pt}
\caption{\label{fig:unstable1psa}
For the Boeing B-767 model, $\sigma_\eps(\BA)$ for the 
full order model (top), compared to the 
unstable ROM (bottom left) and its stabilized variant (bottom right),
both with $k=20$.
(All plots use $\eps=10^{-1}, 10^{-1.25}, \ldots, 10^{-4}$, and 
were computed with Eigtool~\cite{Wri02a}.)
The stabilization procedure repels 
eigenvalues near the origin and reduces the departure from normality.
(These models have a few eigenvalues beyond these axes,
which do not have a major influence on the transient dynamics.)
For a general investigation of the use of orthogonal Krylov projection
to approximate pseudospectra, see~\cite{TT96}.
}
\end{figure}

\begin{example} \label{ex:boeing}

This example begins with a benchmark problem for the control of a
flutter condition in a Boeing B-767 aircraft, contributed by 
Anderson, Ly, and Liu to the collection~\cite{Dav90}.
This $55\times 55$ matrix, call it $\BA_0$, is unstable, having a 
complex-conjugate pair of eigenvalues in the right half-plane.
Burke, Lewis, and Overton~\cite{BLO03a} used an eigenvalue optimization 
algorithm to design a low-rank perturbation that stabilizes $\BA_0$.
It is this stable $\BA$ that we investigate here; we will
refer to it as the ``original model'' when comparing it to the ROMs 
we derive from it.

This $\BA$ has spectral abscissa $\alpha(\BA) \approx  -0.07877$,
but the numerical range $W(\BA)$ extends far into the right half-plane,
with numerical abscissa 
\[ \omega(\BA) \approx 8.4560 \times 10^6.\]
As evident from~(\ref{eq:t0}), solutions to $\dot\Bx(t) = \BA\Bx(t)$ must initially 
exhibit strong growth, although the system is asymptotically stable;
Figure~\ref{fig:unstable1} provides confirmation.
The top plot in Figure~\ref{fig:unstable1psa} shows $\eps$-pseudospectra of $\BA$,
which were previously investigated in~\cite[Chapter~15]{TE05plain}.

It suffices to simply consider $\dot\Bx(t)=\BA\Bx(t)$ here,
rather than the full input-output system~(\ref{eq:xdot})--(\ref{eq:y}).
We let the initial condition
\[ \Bx(0) := \Bx_0 = [1, 1, \ldots, 1]^T/\sqrt{55} \in \C^{55}\]
play the role that the input vector $\Bb$ normally does when
constructing Krylov-based ROMs.

We select the dimension $k=20$ for our ROMs.
(The qualitative results shown here are not particularly sensitive 
to the choice of $k$ and $\Bx_0$.)
First we construct an orthogonal projection model using
the Krylov subspace $\CV=\CK_k(\BA,\Bx_0)$.
The resulting $\VAV$ is unstable: it has 5~eigenvalues in
the right half-plane, as can be seen in the bottom-left plot 
in Figure~\ref{fig:unstable1psa}.
Of course, as Figure~\ref{fig:unstable1} shows, 
the unstable model diverges from the stable model as $t\to\infty$,
\emph{but at earlier times, it does an excellent job of signaling 
the system's transient growth on the scale of $10^3$}.
Such growth could be significant for a motivating application, e.g., warning
that a linearized model might be a poor approximation of an underlying 
nonlinear system.

Still, the instability in $\VAV$ is unappealing, and one might naturally
prefer a stable ROM.\ \  To obtain stability, 
we follow a general approach of Grimme, Sorensen, and van Dooren
(in the context of the bi-Lanczos algorithm~\cite{GSV95}).
To build an orthonormal basis for $\CK_k(\BA,\Bx_0)$
we use the Arnoldi algorithm~\cite{Saa80}.
After constructing the order-20 unstable model,
we \emph{restart} the Arnoldi algorithm~\cite{Sor92},
replacing the starting vector $\Bx_0$ with a
``filtered'' starting vector $\phi_1(\BA)\Bx_0$, where 
$\phi_1$ is the degree-5 monic polynomial with roots 
at the unstable eigenvalues of $\VAV$.
We then use $\CV=\CK_k(\BA,\phi_1(\BA)\Bx_0)$
to construct a new ROM, which now only has only 3~unstable modes.
Repeat the process: let the cubic polynomial $\phi_2$ have roots 
at the 3~new unstable modes, and add $\phi_2(\BA)$ to the filter.
The next space $\CV=\CK_k(\BA,\phi_2(\BA)\phi_1(\BA)\Bx_0)$ has just 1~unstable mode,
which we make the root of the linear polynomial $\phi_3$.
Finally, $\CV=\CK_k(\BA,\Phi(\BA)\Bx_0)$ delivers a stable order-20 ROM, 
where $\Phi(z) := \phi_3(z)\phi_2(z)\phi_1(z)$ has roots at all the previously
encountered unstable modes.

The relevant eigenvalues of this new stabilized ROM can be seen in 
the bottom-right plot of Figure~\ref{fig:unstable1psa}.
Notice that the polynomial filter, with its roots at the unstable eigenvalues
(including the 5~unstable modes in the bottom-left plot of Figure~\ref{fig:unstable1psa}),
effectively deters the ROM from having modes near the origin.
Moreover, the stabilization process has suppressed the departure
from normality, as reflected in a diminished numerical abscissa:  
\begin{center}\begin{tabular}{rr}
original model: & $\omega(\BA)\approx 8.45603 \times 10^6$;\\[.1em]
unstable ROM:   & $\omega(\VAV)\approx 2.24872 \times 10^6$;\\[.1em]
stabilized ROM:   & $\omega(\VAV)\approx 8.90589 \times 10^4$.
\end{tabular}\end{center}
The stabilized system still exhibits transient growth for some initial conditions
(see the bottom plot of Figure~\ref{fig:unstable1}), but far less than the original model.
\end{example}

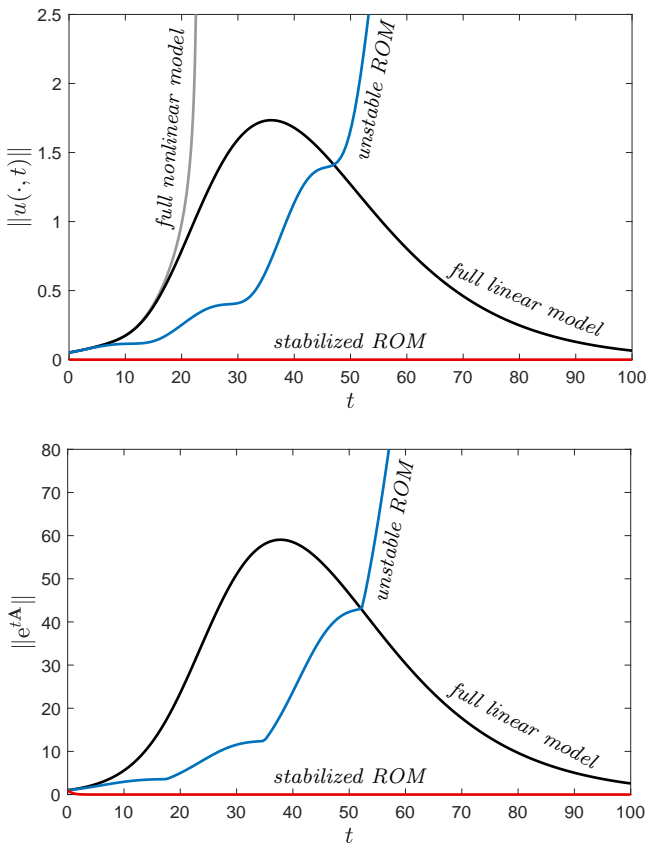
\begin{figure}[b!]
\begin{center}
\includegraphics[scale=.5]{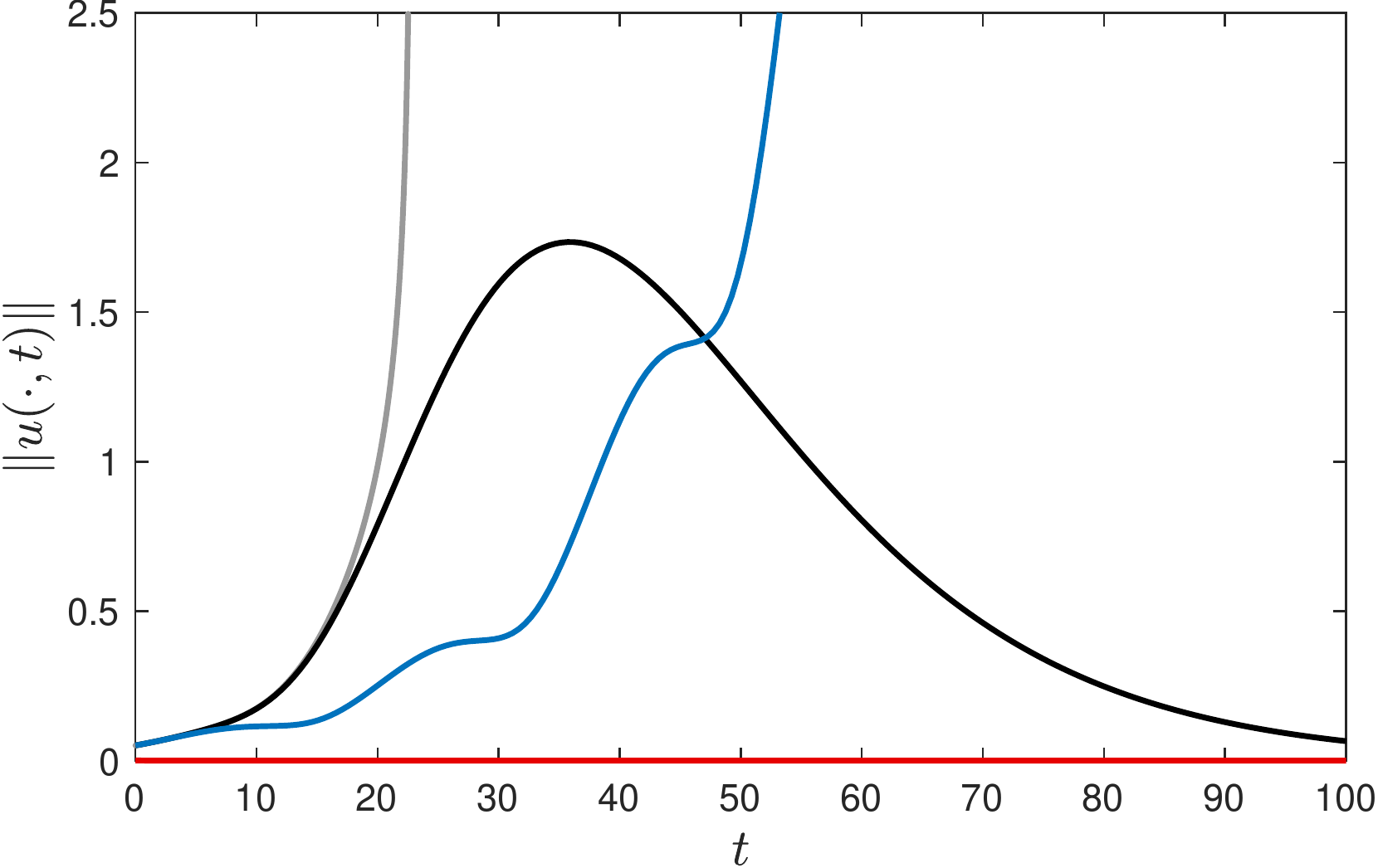}

\begin{picture}(0,0)
\put(-64,80){\rotatebox{86}{\footnotesize \emph{full nonlinear model}}}
\put(10,105){\rotatebox{79}{\footnotesize \emph{unstable ROM}}}
\put(45,63){\rotatebox{-23}{\footnotesize \emph{full linear model}}}
\put(-20,35){\footnotesize \emph{stabilized ROM}}
\end{picture}

\includegraphics[scale=.5]{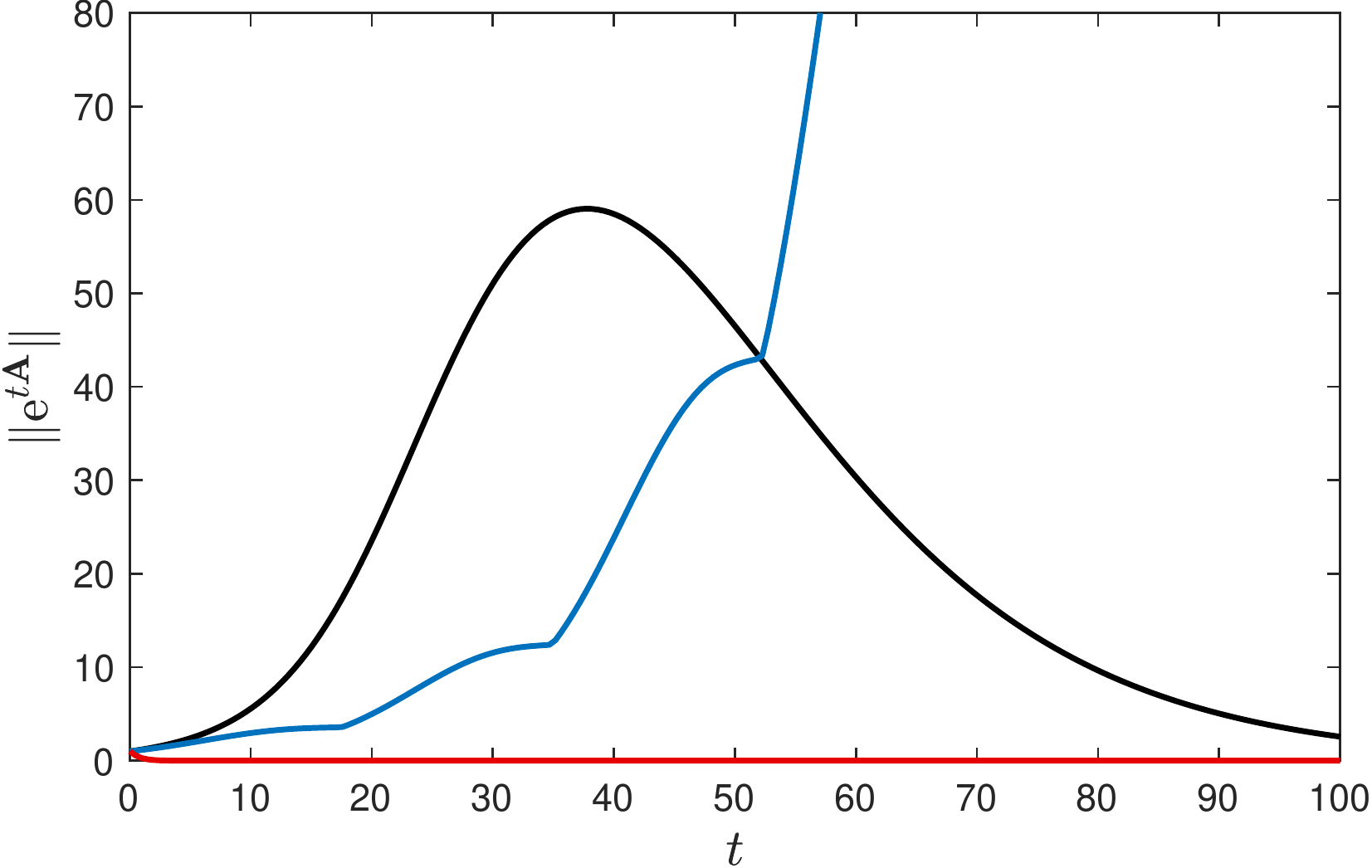}

\begin{picture}(0,0)
\put(17,103){\rotatebox{79}{\footnotesize \emph{unstable ROM}}}
\put(45,69){\rotatebox{-28}{\footnotesize \emph{full linear model}}}
\put(-20,35){\footnotesize \emph{stabilized ROM}}
\end{picture}
\end{center}

\vspace*{-28pt}
\caption{\label{fig:unstable2}
Evolution of a solution $u(x,t)$ to the nonlinear heat model and its 
linearization from Example~\ref{ex:nlheat} (top),
and the analogous plot for the solution operator $\eop^{t\BA}$
for the linearized operator and two linear ROMs derived from it (bottom).
Again, the unstable ROM provides a more accurate impression of 
the dynamics at early $t$ than does the stabilized model 
(both of order $k=40$).
}
\end{figure}

\begin{figure}[b!]
\begin{center}
\includegraphics[scale=.48]{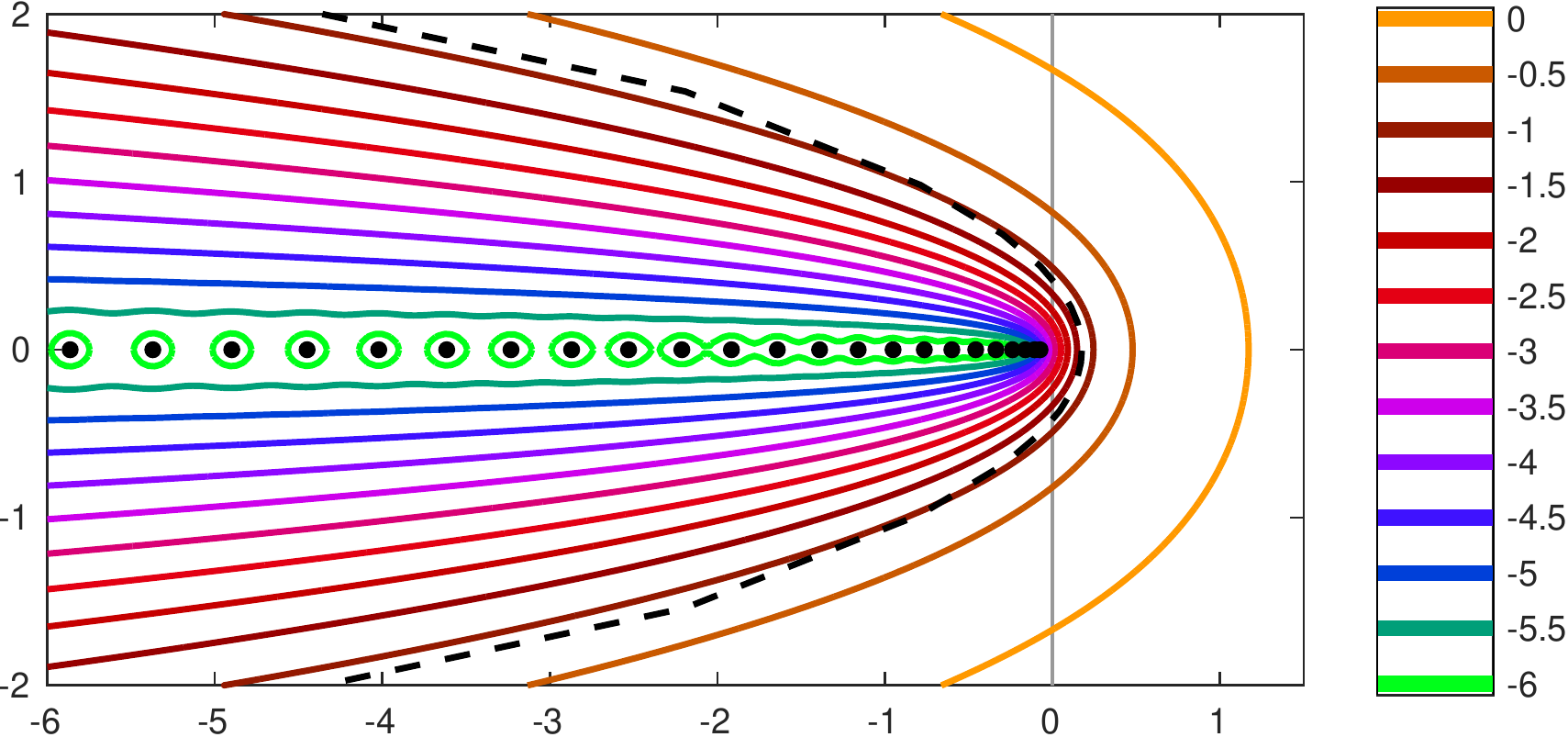}

{\footnotesize \emph{original model}}

\vspace*{9pt}
\includegraphics[scale=.48]{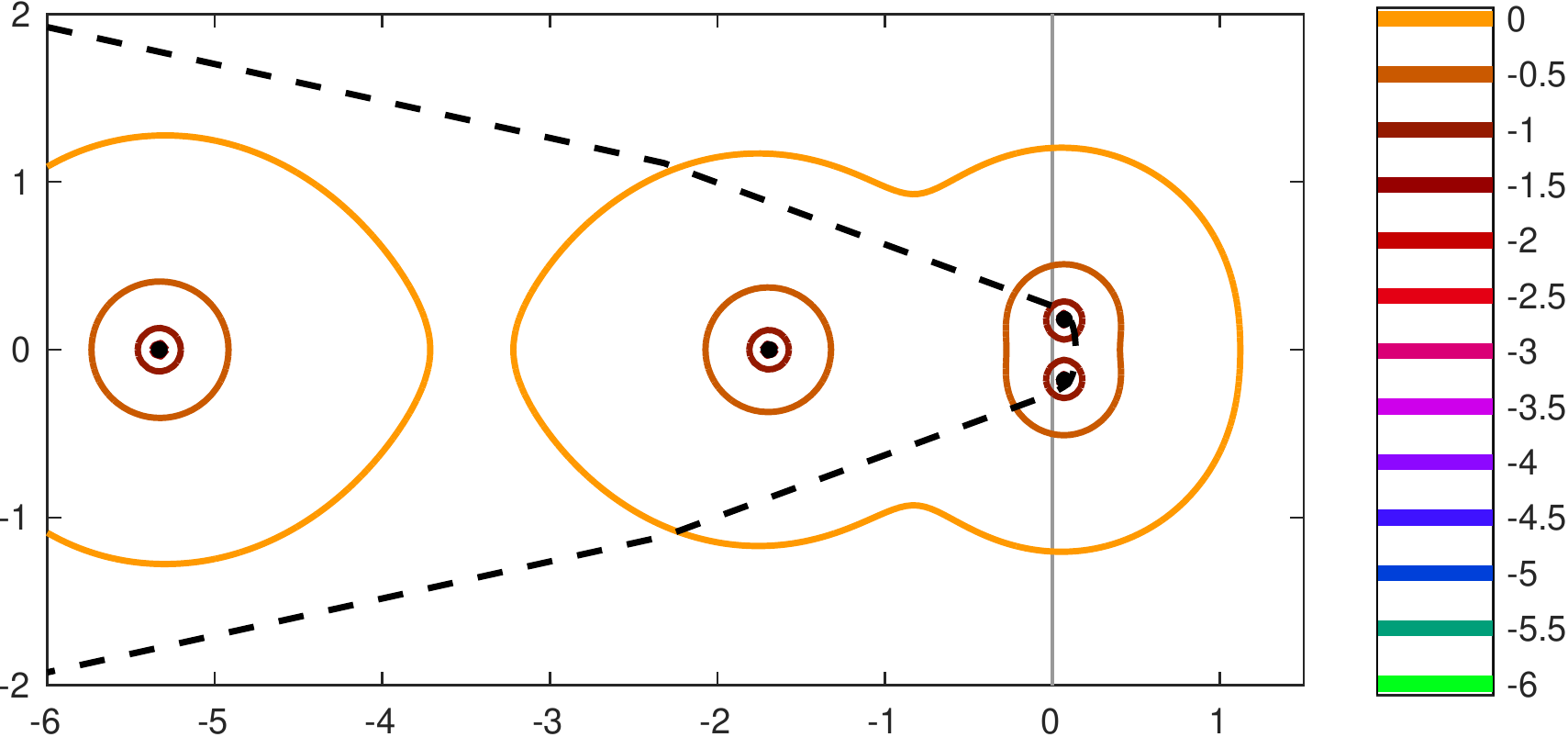}

{\footnotesize \emph{unstable ROM}}

\vspace*{9pt}
\includegraphics[scale=.48]{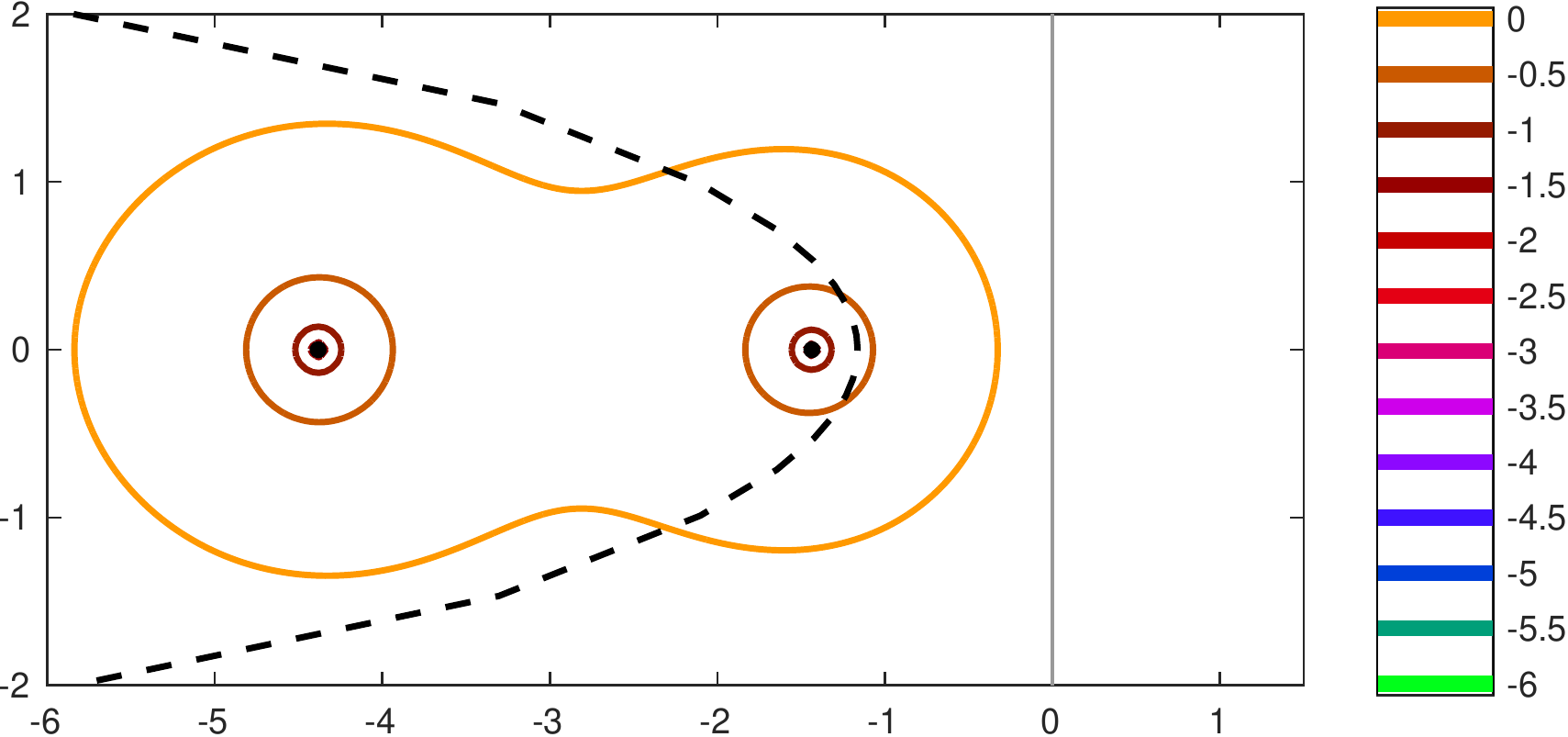}

{\footnotesize \emph{stabilized ROM}}

\end{center}

\vspace*{-14pt}
\caption{\label{fig:unstable2psa}
For the linearized heat model, the rightmost part of 
$\sigma_\eps(\BA)$ for the full order discretization (top)
and order $k=40$ ROMs (middle and bottom).  
(All plots use the same scale $\eps=10^{0}, 10^{-0.5}, \ldots, 10^{-6}$,
and were computed with EigTool~\cite{Wri02a}.)
The dashed black curves denote the boundaries of the numerical range;
the stabilitzation procedure drives both eigenvalues and $W(\VAV)$ too far to the left.
\vspace*{-10pt}
}
\end{figure}

\begin{example} \label{ex:nlheat}
Consider the nonlinear heat equation 
\begin{equation} \label{eq:nlheat}
    u_t(x,t) = u_{xx}(x,t) + u_x(x,t) + \textstyle{1\over 8} u(x,t) + u^3(x,t),
\end{equation}
posed on the domain $x\in(0, \ell) \subset \R$ and $t\ge 0$,
with homogeneous Dirichlet boundary conditions $u(0,t) = u(\ell,t) = 0$.
We are interested in initial conditions $u_0(x) = u(x,0) \in H^1_0(0,\ell)$ 
that are small in norm.
The equation~(\ref{eq:nlheat}) was studied by Sandstede and Scheel~\cite{SS05c}, 
who analyzed stability properties of the trivial solution $u\equiv 0$ 
on the finite domain $x\in (0,\ell)$ versus the infinite domains 
$x \in (0,\infty)$ and $x\in (-\infty,\infty)$.
(On finite domains, sufficiently small initial conditions $u_0$ give solutions 
$u(x,t)$ for which $\|u(\cdot,t)\|_{H^1_0(0,\ell)} \to 0$ as $t\to\infty$;
in contrast, $u\equiv 0$ is unstable on the infinite domains $x\in(0,\infty)$ and
$x\in (-\infty,\infty)$.  This discrepancy suggests that moderately small values of 
$u_0$ can exhibit interesting behavior on finite domains.
Galkowski generalized this model, drawing a connection between the dynamics 
and pseudospectra of the linear part of the model~\cite{Gal12}.)

We take a domain of length $\ell=30$ and discretize the system~(\ref{eq:nlheat})
using a Chebyshev pseudospectral collocation method, based on codes and techniques
from Trefethen~\cite{Tre00}.
Figure~\ref{fig:unstable2} shows the evolution of $\|u(\cdot,t)\|_{L^2(0,\ell)}$ 
for the initial condition
\[ u_0(x,t) = 10^{-5}@x(x-\ell)(x-\ell/2),\]
along with the solution of the analogous linear problem that omits the $u^3$ term
in~(\ref{eq:nlheat}).
The linear model is stable, but experiences transient growth; this growth
gradually increases the contribution of the $u^3$ term in the nonlinear model,
leading to apparent divergence.
Figure~\ref{fig:unstable2psa} shows $\sigma_\eps(\BA)$ (in the $L^2(0,\ell)$ norm)
for this discretized, linearized operator of order $n=127$.
To obtain a ROM for this linearized part of the problem, we transform coordinates
so the vector Euclidean norm approximates the $L^2(0,\ell)$ norm, 
then compute the associated Krylov orthogonal projection ROM.\ \ 
The resulting order $k=40$ ROM has a complex conjugate pair of 
unstable eigenvalues.
Following the stabilization procedure described in Example~\ref{ex:boeing}, 
one iteration of the restarted Arnoldi method with a filtered starting vector
yields a stabilized model.
As with Example~\ref{ex:boeing}, the unstable ROM does a better job of capturing the 
transient growth of the linear system.  While the unstable ROM does not qualitatively
match the asymptotic behavior of the \emph{linear system}, it is more consistent with 
the apparent divergence of the true \emph{nonlinear system}.

\end{example}

\vspace*{-12pt}

\section{Conclusion}
This note has collected a number of results that give insight into the 
unstable modes that can arise in projection-based reduced-order models.
The illustrative examples have been intentionally small in scale to make
simple points, but the implications for large-scale systems are evident.
While the bounds in Theorems~\ref{thm:carden} and~\ref{thm:carden2} 
limit the location and number of unstable modes, these bounds can 
undoubtedly be sharpened with further analysis.
\vspace*{-7pt}
\section*{Acknowledgments}

I thank Jurjen Duintjer Tebbens, Serkan Gugercin, 
and Dan Sorensen for helpful conversations,
and Laurent Demanet for introducing me to the
nonlinear heat model in Example~\ref{ex:nlheat}.
I appreciate a referee's insightful suggestions.



\end{document}